\documentclass[12pt]{amsart}
\usepackage{txfonts, color}
\usepackage{mathrsfs}
\usepackage{amsmath}
\usepackage{amssymb, amsmath}
\usepackage{graphicx}
\newcommand{\newcom}{\newcommand}

\newcom{\al}{\alpha}
\newcom{\be}{\beta}
\newcom{\eps}{\epsilon}
\newcom{\veps}{\varepsilon}
\newcom{\e}{\varepsilon}
\newcom{\ga}{\gamma}
\newcom{\Ga}{\Gamma}
\newcom{\ka}{\kappa}
\newcom{\Lam}{\Lambda}
\newcom{\lam}{\lambda}
\newcom{\Om}{\Omega}
\newcom{\om}{\omega}
\newcom{\Si}{\Sigma}
\newcom{\si}{\sigma}
\newcom{\tht}{\theta}
\newcom{\dtri}{\nabla}
\newcom{\tri}{\triangle}
\newcom{\oo}{\infty}
\newcom{\vphi}{\varphi}
\newcom{\cB}{{\mathcal B}}
\newcom{\cC}{{\mathcal C}}
\newcom{\cD}{{\mathcal D}}
\newcom{\cF}{{\mathcal F}}
\newcom{\cH}{{\mathcal H}}
\newcom{\cL}{{\mathcal L}}
\newcom{\cM}{{\mathcal M}}
\newcom{\cN}{{\mathcal N}}
\newcom{\cP}{{\mathcal P}}
\newcom{\cS}{{\mathcal S}}
\newcom{\cQ}{{\mathcal Q}}
\newcom{\cT}{{\mathcal T}}
\newcom{\cY}{{\mathcal Y}}
\newcom{\cZ}{{\mathcal Z}}
\newcom{\R}{\mathbb R}
\newcom{\T}{\mathbb T}
\newcom{\BT}{{\mathbb{T}^2}}
\newcom{\Z}{\mathbb Z}
\newcom{\C}{\mathbb C}
\newcom{\E}{\mathbb E}

\newcom{\hha}{\hat{\mathbf h}}
\newcom{\ha}{\hat{h}}
\newcom{\ul}{\underline}
\newcommand{\vc}[1]{{\mathbf #1}}
\newcom{\ve}{\vc{e}}
\newcom{\vN}{\vc{N}}
\newcom{\vn}{\vc{n}}
\newcom{\vG}{\vc{G}}
\newcom{\vF}{\vc{F}}
\newcom{\vf}{\vc{f}}
\newcom{\vg}{\vc{g}}
\newcom{\vq}{\vc{q}}
\newcom{\vu}{\vc{u}}
\newcom{\vv}{\vc{v}}
\newcom{\vw}{\vc{w}}
\newcom{\vb}{\vc{b}}
\newcom{\vh}{\vc{h}}
\newcom{\vz}{\vc{z}}
\newcom{\vup}{\vu^{+}}
\newcom{\vum}{\vu^{-}}
\newcom{\vvp}{\vv^{+}}
\newcom{\vvm}{\vv^{-}}
\newcom{\vbp}{\vb^{+}}
\newcom{\vbm}{\vb^{-}}
\newcom{\vhp}{\vh^{+}}
\newcom{\vhm}{\vh^{-}}
\newcom{\Omp}{{\Om^+}}
\newcom{\Omm}{{\Om^-}}
\newcom{\vupm}{{\vu^{\pm}}}
\newcom{\vvpm}{{\vv^{\pm}}}
\newcom{\vbpm}{{\vb^{\pm}}}
\newcom{\vhpm}{{\vh^{\pm}}}
\newcom{\vwp}{{\vc{w}^+}}
\newcom{\vwm}{{\vc{w}^-}}
\newcom{\vwpm}{{\vc{w}^{\pm}}}
\newcom{\Ompm}{{\Omega^{\pm}}}
\newcom{\vom}{\boldsymbol{\omega}}
\newcom{\vvap}{\vc{\varpi}}
\newcom{\vop}{\vom^{+}}
\newcom{\vnu}{\vc{\nu}}
\newcom{\vopm}{\vom^{\pm}}
\newcom{\vjp}{\vj^+}
\newcom{\vjm}{\vj^-}
\newcom{\vjpm}{\vj^{\pm}}
\newcom{\vj}{\boldsymbol{\xi}}
\newcom{\Ds} {\langle\nabla\rangle^{s-\f12}}

\newcom{\ds}{{\rm d} s}
\newcom{\f}{\frac}
\newcom{\di}{\displaystyle\int}
\newcom{\dl}{\displaystyle\lim}
\newcom{\ov}{\overline}
\newcom{\sset}{\subset}
\newcom{\wt}{\widetilde}
\newcom{\pa}{\partial}
\newcom{\p}{\partial}
\newcom\na{\nabla}
\newcom{\suml}{\sum\limits}
\newcom{\supl}{\sup\limits}
\newcom{\intl}{\int\limits}
\newcom{\infl}{\inf\limits}
\newcom{\disp}{\displaystyle}
\newcom{\non}{\nonumber}
\newcom{\no}{\noindent}
\newcom{\QED}{$\square$}

\def\div{\mathop{\rm div}\nolimits}
\def\curl{\mathop{\rm curl}\nolimits}

\def\eqdefa{\buildrel\hbox{\footnotesize def}\over =}

\newtheorem{athm}{\bf \t}[section]
\newenvironment{thm} [1] {\def\t{#1}\begin{athm} \bf \rm} {\end{athm}}
\newcom{\bthm}{\begin{thm}}\newcom{\ethm}{\end{thm}}

\newtheorem{theorem}{Theorem}[section]
\newtheorem{lemma}{Lemma}[section]
\newtheorem{remark}{Remark}[section]

\newtheorem{definition}{Definition}[section]
\newtheorem{proposition}{Proposition}[section]

\newcom{\beq}{\begin{equation}}
\newcom{\eeq}{\end{equation}}
\newcom{\ben}{\begin{eqnarray}}
\newcom{\een}{\end{eqnarray}}
\newcom{\beno}{\begin{eqnarray*}}
\newcom{\eeno}{\end{eqnarray*}}
\newcom{\bali}{\begin{aligned}}
\newcom{\eali}{\end{aligned}}

\numberwithin{equation}{section}

\begin{document}

\title[Well-posedness of plasma-vacuum interface problem]{Well-posedness of the plasma-vacuum interface problem for ideal incompressible MHD}

\author{Yongzhong Sun}
\address{Department of Mathematics, Nanjing University, 210093, Nanjing, P. R. China}
\email{sunyz@nju.edu.cn}

\author{Wei Wang}

\address{School of Mathematical Sciences, Zhejiang University, 310027, Hangzhou, P. R. China}
\email{wangw07@zju.edu.cn}

\author{Zhifei Zhang}
\address{School of Mathematical Sciences, Peking University, 100871, P. R. China}
\email{zfzhang@math.pku.edu.cn}

%\date{\today}%4.22 2017

\begin{abstract}
  In this paper, we prove the local well-posedness of plasma-vacuum interface problem for ideal incompressible magnetohydrodynamics under the stability condition: the magnetic field $\vh$ and the vacuum magnetic field $\hat\vh$ are non-collinear on the interface(i.e., $|\vh\times \hat \vh|>0$), which was introduced by Trakhinin as a stability condition for the compressible plasma-vacuum interface problem.
\end{abstract}

\maketitle

%\tableofcontents

\section{Introduction}\label{s1}

\subsection{Presentation of the problem}

Magnetohydrodynamic (MHD) models describe macroscopic plasma phenomena, from laboratory research on thermonuclear fusion to plasma-astrophysics of the solar system. In the laboratory research, the main topic is magnetic plasma confinement for energy production by controlled thermonuclear reactions. The plasma-vacuum interface appears as a typical phenomenon when the plasma is separated by a vacuum from outside wall. The total pressure is balanced on the interface, while the normal part of the magnetic field vanishes and the tangent part may jump, thus forms a tangential discontinuity. Mathematically the plasma-vacuum interface is formulated as a free boundary problem for the MHD system, see for example \cite{GP}. By ignoring the viscosity, the resistivity and heat conduction, we assume that the plasma fluid is ideal and incompressible. The evolution of the velocity $\vu=(u^1, u^2, u^3)$, the magnetic field $\vh=(h_1,h_2,h_3)$ and the total  pressure $p$ is formulated by the following system of partial differential equations:
\beno
\left\{
\begin{array}{l}
\p_t \vu + \vu\cdot\nabla\vu - \vh\cdot\nabla\vh + \nabla p= 0, \\
\div \vu = 0,\quad \div\vh=0, \\
\p_t \vh + \vu\cdot\nabla\vh - \vh\cdot\nabla\vu = 0.
\end{array}\right.
\eeno
Here the total pressure $p=q + \frac{1}{2}|\vh|^2$ with $q$ the fluid pressure. For technical reason, we consider the plasma-vacuum interface problem under a simplified configuration. Denote $\Om = \mathbb{T}^2 \times (-1,1)$ with the top/bottom boundary $\Gamma^{\pm}= \mathbb{T}^2 \times \{\pm 1\}$ and assume that the plasma is initially confined in the domain
\[
\Om^{-}_{f_0}=\left\{ x=(x',x_3) \in \Om | x_3 < f_0(x')\right\}, \, x'=(x_1,x_2)\in \mathbb{T}^2,
\]
where $f_0(x')$ is a function defined on $\mathbb{T}^2$ and
\[
\Gamma_{f_0} := \left\{x\in \Om | x_3=f_0(x'), x' \in \mathbb{T}^2 \right\}
\]
is the initial interface. Consequently,
\[
\Om^{+}_{f_0}=\left\{ x=(x',x_3) \in \Om | x_3 > f_0(x')\right\}
\]
is the region of the initial vacuum. After the initial time, the plasma evolves and the interface moves simultaneously. At the time $t>0$, let us assume the interface is represented as
\[
\Gamma_f = \Gamma_{f(t)} := \left\{x\in \Om | x_3=f(t,x'), x' \in \mathbb{T}^2 \right\},
\]
and denote
\[
\Om^-_f = \Om^{-}_{f(t)}=\left\{ x=(x',x_3) \in \Om | x_3 < f(t,x')\right\}, \quad Q^-_T :=\cap_{t\in (0,T)}\{t\}\times \Om^-_f,
\]
\[
\Om^{+}_f =  \Om^{+}_{f(t)}=\left\{ x=(x',x_3) \in \Om | x_3 > f(t,x')\right\}, \quad Q^+_T :=\cap_{t\in (0,T)}\{t\}\times \Om^+_f.
\]

With these notations, the evolution of the plasma part is formulated as
\ben\label{mhd01}
&&\p_t \vu + \vu\cdot\nabla\vu - \vh\cdot\nabla\vh + \nabla p= 0\quad \text{in}\quad Q^-_T,\\
\label{mhd02} &&\div \vu = 0,\,\, \div\vh=0\quad \text{in}\quad Q^-_T,\\
\label{mhd03} &&\p_t \vh + \vu\cdot\nabla\vh - \vh\cdot\nabla\vu = 0\quad \text{in}\quad Q^-_T.
\een
In the vacuum domain $\Om^+_f$, we consider so-called {\it pre-Maxwell dynamics}. In such case, the magnetic field $\hat\vh=(\hat{h}_1,\hat{h}_2,\hat{h}_3)$ is
determined by the div-curl system:
\beq\label{mhd1}
\div\hha=0, \quad \curl \hha=0\quad \text{in}\quad \Omega^+_f.
\eeq

The physical quantities of the plasma and the vacuum region are related by the pressure balance condition on the interface $\Gamma_f$:
\beq\label{mhd2}
p(t,x)=\frac12|\hat{\vh}(t,x)|^2, \quad (t,x)\in \Gamma_f
\eeq
as well as
\beq\label{mhd3}
\vh\cdot\vN_f = 0, \quad \hat\vh\cdot\vN_f = 0, \quad (t,x)\in \Gamma_f.
\eeq
Here
\[
\vN_f = (N_1,N_2,N_3) = (-\p_1 f, -\p_2 f, 1)
\]
is the normal vector of $\Gamma_f$. As the interface moves with the fluid particles, its normal velocity $\p_t f$ satisfies
\beq\label{mhd4}
\p_t f = \vu\cdot\vN_f.
\eeq
Moreover, on the artificial boundaries $\Gamma^{\pm}$, we prescribe the following boundary conditions:
\ben\label{mhd5}
&&u_3=0,\,\, h_3=0\quad \text{on}\quad\Gamma^-,\\
\label{mhd6}
&&\hat{\vh}\times \vc{e}_3 = \hat{\mathbf{J}}\quad\text{on}\quad\Gamma^+,
\een
where $\vc{e}_3 = (0,0,1)$ is the unit normal vector on $\Gamma^+$. Here to avoid trivial solution $\hha$ in the vacuum, a surface current $\hat{\vc{J}}=({\hat{J}}_1,{\hat{J}}_2,{\hat{J}}_3)$ is added as an outer force term to the elliptic system (\ref{mhd1}).
In real laboratory plasma, this surface current can be produced by a system of coils, see \cite{GP,ST}. Finally, the system is supplemented with the initial data
\beq\label{mhd7}
\vu(0,x) = \vu_{0}(x), \quad \vh(0,x)=\vh_{0}(x)\,\,\text{ in }\,\, \Om^-_{f_0}, \quad f(0,x_1,x_2) = f_0(x_1,x_2),
\eeq
which satisfy the following compatibility conditions:
\beq\label{mhd8}
\left\{\begin{aligned}
&\div \vu_{0} = 0, \,\, \div\vh_{0} = 0\,\, \text{ in } \Om_{f_0}^-,\,\,  \\
&\vh_{0}\cdot\vN_{f_0} = 0\,\, \text{ on } \Gamma_{f_0},\quad u_{30},\, h_{30} = 0\,\,\text{ on }\Gamma^- .
\end{aligned}\right.
\eeq
From (\ref{mhd6}) and the fact that $\curl \hha=0$, we also need compatibility conditions on the imposed surface current:
\begin{align}\label{mhd8-0}
  \p_1\hat{J}_1 + \p_2\hat{J}_2 = 0, \,\,\hat{J}_3 = 0\,\,\text{ on }\,\Gamma^+.
\end{align}

Note that the initial magnetic field $\hat\vh_0$ in the vacuum region is uniquely determined from $\Gamma_{f_0}, \vu_0,\vh_0$ and $\hat{\vc{J}}_0= \hat{\vc{J}}(0,x')$ by solving the following div-curl system:
\beq\label{mhd9}
\left\{\begin{array}{l}
\curl \hat\vh_0 =0,\quad \div \hat\vh_0 = 0\,\, \text{ in } \Om_{f_0}^+,\\
\hat\vh_0\cdot\vN_{f_0} =0 \text{ on } \Gamma_{f_0},   \quad
\hat\vh_0 \times \ve_3 = \hat{\vc{J}}_0\,\, \text{ on } \Gamma^{+}.
\end{array}\right.
\eeq
The solvability of this div-curl system will be shown in Section \ref{s4}. Also note that since
\[
\p_t\div\vh + \vu\cdot\nabla\div\vh=0,
\]
the divergence free restriction on $\vh$ is automatically satisfied if it holds for $\vh_0$. Similar argument also yields $\vh\cdot\vN_f = 0$ provided $\vh_{0}\cdot\vN_{f_0} = 0$.

\subsection{Backgrounds}

In the absence of the magnetic field, the system is reduced to the incompressible Euler equations with a free boundary, which is so-called water-wave problem.
In this case, it is well-known that a sufficient condition ensuring the well-posedness of the water-wave problem is the Taylor's sign condition:
\ben\label{Taylor}
\frac{\p p}{\p \vN} \leq - \epsilon <0\quad \text{ on }\,\,\Gamma_f.
\een
See \cite{ABZ, CS, Lan, Lin, SZ, Wu1, Wu2, ZZ} and references therein. In fact, it is also necessary in the absence of surface tension \cite{Ebin}.

The fact that the magnetic field has the stabilizing effect for the current-vortex sheet problem was found by physicists long before \cite{Ax, Sy}.
In past decade, there are many works devoted to the well-posedness of the current-vortex sheet problem under a suitable stability condition \cite{CW,CMST,SWZ,Tra1,Tra2,WY}.

In \cite{Tra-JDE}, Trakhinin introduced the following stability condition for the linearized compressible plasma-vacuum interface problem:
\beq\label{stab0}
|\vh\times\hat\vh| >0\, \text{ on }\Gamma_f.
\eeq
Under this condition, Secchi and Trakhinin \cite{ST} proved the well-posedness of the compressible plasma-vacuum interface problem, and
Morando, Trakhinin and Trebeschi \cite{MTT} proved the well-posedness of the linearized incompressible plasma-vacuum interface problem.
However, the well-posedness of nonlinear incompressible problem \big(i.e., the system (\ref{mhd01})-(\ref{mhd7})\big) is still open under \eqref{stab0}.

Motivated by works on the water-wave problem, Luo and Hao \cite{Hao, HL} established \emph{a priori} estimates for the incompressible plasma-vacuum interface problem under the Taylor's sign condition \eqref{Taylor}. Recently, Gu and Wang \cite{GW} proved the well-posedness of the incompressible plasma-vacuum problem under \eqref{Taylor}. Let us also mention that
the well-posedness of the plasma-vacuum interface problem under \eqref{Taylor} is still unknown when the vacuum magnetic field $\hat\vh$ is non trivial.
In \cite{Hao, HL, GW}, the authors only considered the case of $\hat\vh=0$. This problem is also unsolved in the compressible case \cite{Tra-JDE} .

\subsection{Main result}
The goal of this paper is to prove the well-posedness of the system (\ref{mhd01})-(\ref{mhd7}) under the stability condition \eqref{stab0}.
This condition implies that $\vh$ and $\hat\vh$ are non-collinear everywhere on $\Gamma_f$, which means
\beq\label{stab1}
\inf_{x\in\Gamma_f}\inf_{\vq\in T_{\Gamma_f}(x),|\vq|=1} \left[(\vh(x)\cdot\vq)^2+(\hat\vh(x)\cdot\vq)^2\right] >0.
\eeq
It will be shown in Section \ref{s5} that (\ref{stab1}) implies that there exists a positive constant $c_1$ such that
\ben\label{con:stab2}
\Lambda(\vh,\hat\vh)\eqdefa\inf_{x\in\Gamma_f}\inf_{\varphi_1^2+\varphi_2^2=1}\left[(h_1\varphi_1+h_2\varphi_2)^2+(\hat{h}_1\varphi_1+\hat{h}_2\varphi_2)^2\right]\ge c_1.
\een

Our main result is stated as follows.

\begin{theorem}\label{main}
Let $s\ge 3$ be an integer. Assume that
\beno
&&f_0\in H^{s+\f12}(\BT),\,  \vu_0,\,\vh_0\in H^{s}(\Omega_{f_0}^-),\\
&&\hat{\mathbf{J}}\in C\big([0,T_0];H^{s-\f12}(\BT)\big),\quad \p_t\hat{\mathbf{J}} \in C\big([0,T_0];H^{s-\f32}(\BT)\big),
\eeno
which satisfies the compatibility conditions  (\ref{mhd8})-(\ref{mhd8-0}) and the stability condition:
\beq\label{sta4}
-(1-2c_0)\le f_0\le (1-2c_0), \quad  \Lambda(\vh_0,\hat\vh_0)\ge 2c_1
\eeq
for some $c_0\in (0,\f12)$ and $c_1>0$.
Then there exists $T\in (0,T_0)$ such that the system (\ref{mhd01})-(\ref{mhd7}) admits a unique solution $(f, \vu, \vh, \hat{\vh})$ in $[0,T]$ such that
\ben\label{mhd11}
&&f\in L^\infty(0,T; H^{s+\f12}(\BT)), \,
\vu,\,\vh \in L^\infty(0,T;H^{s}(\Omega^-_f)), \,\hat{\vh}\in L^\infty(0,T;H^{s}(\Omega^+_f)),\\
\label{mhd12}
&&-(1-c_0)\le f\le (1-c_0), \,\,
\Lambda(\vh,\hat{\vh})\ge c_1.
\een
\end{theorem}

The idea of the proof is motivated by our recent work on nonlinear stability of incompressible current-vortex sheet problem \cite{SWZ}.
The key ingredient is to derive an evolution equation of the scaled normal velocity $\vu\cdot\vN$ rather than the
usual normal component of velocity on the interface. By some tricky observations, we find in present case that the evolution equation of the height function of the interface takes as follows
\begin{align}\label{eq:theta-intro}
D_t^2f=\sum_{i,j=1,2}
\big(\ul{h_i}\ul{h_j}
+\ul{\hat h_i}\ul{\hat h_j}\big)\partial_i\partial_jf+\text{{L.O.T}.},
\end{align}
where $D_tf=\pa_tf+\ul{u_1}\pa_1f+\ul{u_2}\pa_2f$, $\ul{g}$ denotes the trace of $g$ on the interface, and {L.O.T.} denotes the lower order terms.
Now, the stability condition \eqref{con:stab2} ensures that
the equation \eqref{eq:theta-intro} is strictly hyperbolic. Indeed, the principal symbol of the operator
$-\sum_{i,j=1,2}
\big(\ul{h_i}\ul{h_j}
+\ul{\hat h_i}\ul{\hat h_j}\big)\partial_i\partial_j$ is
\beno
(\ul{h_i}\xi_i)^2+(\ul{\hat h_i}\xi_i)^2\ge c_1|\xi|^2.
\eeno

The motion of the fluid and the magnetic field will be described by the following vorticity and current system:
\begin{align*}
\left\{\begin{array}{l}
\p_t\vom+\vu\cdot\nabla\vom-\vh\cdot\nabla\vj=\vom\cdot\nabla\vu-\vj\cdot\nabla\vh\quad \text{ in }\,\,Q_T^-,\\
\p_t\vj+\vu\cdot\nabla\vj-\vh\cdot\nabla\vom=\vj\cdot\nabla\vu
-\vom\cdot\nabla\vh-2\nabla u_i\times\nabla h_i\quad \text{ in }\,\,Q_T^-.
\end{array}\right.
\end{align*}
where $\vom = \nabla\times\vu,\, \vj = \nabla\times\vh$. With the vorticity and current, the velocity and magnetic field can be recovered by solving two div-curl systems defined on the fluid domain and the vacuum domain respectively. For the fluid domain, we need to solve a div-curl system with given normal component
on the boundaries, which is the same as in \cite{SWZ}. For the vacuum domain, we need to solve  another div-curl system for $\hat\vh$:
\beno\left\{\begin{aligned}
&\div \hat\vh= 0,\quad \curl \hat\vh = 0\,\,\text{ in } \Om_f^+,\\
&\hat\vh\cdot\vn_f = \hat\vartheta \text{ on } \Gamma_f,\quad
\hat\vh\times\ve_3 = \hat{\mathbf{J}}\,\, \text{ on } \Gamma^{+}.
\end{aligned}\right.\eeno
The main difference of these two system is the boundary condition on the fixed boundary.
We state the solvability and estimates of solutions to the div-curl systems in Section 4.

The construction of the approximate solution is completed by introducing the suitable linearization of the system and the iteration map. It can be proved that the constructed approximate sequence is a Cauchy sequence, and the limit system is equivalent to the origin system.
\medskip

The rest of this paper is organized as follows. In Section \ref{s2}, we introduce the harmonic coordinate and preliminary results on the harmonic extensions and Dirichlet to Neumann operators. In Section \ref{s3}, the system is replaced by a new formulation.
In Section \ref{s4}, two div-curl systems are solved. In Section \ref{s5}, we establish the uniform estimates for the linearized system. Section \ref{s6}-\ref{s8} are devoted to proving existence (and uniqueness) of the solution.

\section{Harmonic Coordinate and Dirichlet-Neumann Operator}\label{s2}
In this section, we recall some facts and well-known results on the reference domain, the harmonic coordinate and Dirichlet-Neumann operators.

We first introduce some notations used throughout this paper. The coordinate in the fluid region
is denoted as $x=(x_1,x_2,x_3)$ or $y=(y_1,y_2,y_3)$, while $x'=(x_1,x_2)$ or $y'=(y_1, y_2)$ is the natural coordinates on the interface or on the top/bottom boundary $\Gamma^{\pm}$. For a function $g:\Omega\to\mathbb{R}$, we denote $\nabla g=(\partial_1g,\partial_2g,\partial_3g)$ and for a function $\eta:\BT\to\mathbb{R}$,
$\nabla\eta=(\partial_1\eta,\partial_2\eta)$. The trace on $\Gamma_f$ for a function $g:\Omega_f^\pm\to\mathbb{R}$ is denoted by $\ul{g}$.
Thus, for $i=1,2$,
\beq\label{res}
(\partial_{i}\ul{g})(x')=\partial_ig(x',f(x'))+\partial_3g(x',f(x'))\partial_if(x') = \ul{\p_i g} + \ul{\p_3g}\p_if(x').
\eeq
Finally, the Sobolev norm in $\Om^{\pm}_f$ is denoted as $\|\cdot\|_{H^k(\Om^{\pm}_f)}$ and $\|\cdot\|_{H^k}$ is the Sobolev norm in $\T^2$.

To solve the free boundary problem, we introduce a fixed reference domain. Let $\Gamma_*$ be a fixed graph given by
\begin{align*}
\Gamma_*=\Big\{(y_1,y_2,y_3):y_3=f_*(y_1,y_2)\Big\}.
\end{align*}
The reference domain $\Om_*^{\pm}$ is given by
\begin{align*}
\Om_*=\mathbb{T}^2\times(-1,1),\quad\Om_*^{\pm}=\Big\{ y \in \Om_*| y_3 \gtrless f_*(y')\Big\}.
\end{align*}

We will seek a free boundary lying in a neighborhood of the reference domain. To this end, we define
\begin{align*}
\Upsilon(\delta,k)&\eqdefa\Big\{f\in H^k(\mathbb{T}^2): \|f-f_*\|_{H^k(\mathbb{T}^2)}\le \delta \Big\}.
\end{align*}
For $f\in \Upsilon(\delta,k)$, the graph $\Gamma_f$ is defined as
\[
\Gamma_f\eqdefa\left\{x\in \Om | x_3=f(t,x'), x' \in \mathbb{T}^2 \right\},
\]
and use the notations $\Om^\pm_f, \vN_f,\Gamma^\pm$, etc., as in Section \ref{s1}.

\begin{remark}
Since we intend to solve the plasma-vacuum interface problem \emph{locally in time}, a natural choice of $\Gamma_*$ would be certainly the initial interface $\Gamma_0$.
\end{remark}

To handle the plasma-vacuum interface problem, we need to introduce different Dirichlet-Neumman operators on $\Om^\pm_f$. For a function $\psi(x')=\psi(x_1,x_2)\in H^s(\mathbb{T}^2)$, its harmonic extension from $\Gamma_f$ to $\Omega^\pm_f$
is denoted as $\mathcal{H}_f^\pm \psi$, i.e.,
\beq\left\{\begin{aligned}
&\Delta \mathcal{H}_f^\pm \psi =0 \text{ in }\Omega_f^\pm,\\
&(\mathcal{H}_f^\pm \psi)(x',f(x'))=\psi(x'), \, x'\in\mathbb{T}^2,\\
&\partial_3\left(\mathcal{H}_f^\pm\psi\right) (x',\pm 1)=0, \, x'\in\mathbb{T}^2.
\end{aligned}\right.
\eeq
Moreover, for a function $g$ defined in $\Om^\pm_f$, we denote $\widehat{\mathcal{H}}^\pm_f g$ as a harmonic function in $\Om^\pm_f$ such that
\beq\left\{\begin{aligned}
&\Delta \widehat{\mathcal{H}}^\pm_f g =0 \text{ in } \Omega_f^\pm,\\
&(\widehat{\mathcal{H}}^\pm_f g )(x',f(x'))=g(x',f(x')), \, x'\in\mathbb{T}^2,\\
&\partial_3\left(\widehat{\mathcal{H}}^\pm_f g\right) (x',\pm1)=\p_3 g(x',\pm1), \, x'\in\mathbb{T}^2.
\end{aligned}\right.
\eeq
The Dirichlet-Neumann (D-N in the following context) operators are defined as
\begin{align}
\mathcal{N}^\pm_f \psi = \mp\vN_f\cdot \nabla\mathcal{H}^\pm_f\psi \big|_{\Gamma_f}, \quad \widehat{\mathcal{N}}^\pm_f g =  \mp \vN_f\cdot\nabla\widehat{\mathcal{H}}^\pm_f g \big|_{\Gamma_f}.
\end{align}
For a function $g$ defined in $\Om_f^+$, we denote
\begin{align}
\overline{\mathcal{N}}_f g = \widehat{\mathcal{N}}^+_f g - \mathcal{N}_f^- \underline{g}.
\end{align}

Given $f\in \Upsilon(\delta,k)$, a map (harmonic coordinate) $\Phi_f^\pm$ from $\Omega_*^\pm$ to $\Omega_f^\pm$
is defined by harmonic extension:
\beq\label{sys:har-coordinate}\left\{\begin{aligned}
&\Delta_y \Phi_f^\pm=0   \text{ in } \Omega_*^\pm,\\
&\Phi_f^\pm(y',f_*(y'))=(y',f(y')), \, y'\in\mathbb{T}^2,\\
&\Phi_f^\pm(y',\pm1)=(y',\pm1), \, y'\in\mathbb{T}^2.\\
\end{aligned}\right.
\eeq
Given $\Gamma_*$, there exists $\delta_0=\delta_0(\|f_*\|_{W^{1,\infty}})>0$ so that $\Phi_f^\pm$ is a bijection whenever $\delta\le \delta_0$.
Thus, there exists an inverse map $\Phi_f^{\pm-1}$ from $\Omega_f^\pm$ to $\Omega_*^\pm$ such that
\beno
\Phi_f^{\pm-1}\circ\Phi_f^\pm=\mathrm{Id}_{\Om^\pm_*}, \, \Phi_f^\pm\circ\Phi_f^{\pm-1}=\mathrm{Id}_{\Om^\pm_f}.
\eeno

We have the following properties of harmonic coordinates, see \cite{SWZ} for example.

\begin{lemma}\label{lem:basic}
Let $f\in \Upsilon(\delta_0,s-\f12)$ for $s\ge 3$. There exists a constant $C$ depending only on $\delta_0$ and  $\|f_*\|_{H^{s-\f12}}$ so that
\begin{itemize}
\item[1.] If $u\in H^{\sigma}(\Om_f^\pm)$ for $\sigma\in [0,s]$, then
\beno
&&\|u\circ\Phi_f^\pm\|_{H^\sigma(\Om^\pm_*)}\le C\|u\|_{H^\sigma(\Om_f^\pm)}.
\eeno
\item[2.] If $u\in H^{\sigma}(\Om_*^\pm)$ for $\sigma\in [0,s]$, then
\beno
\|u\circ\Phi_f^{\pm-1}\|_{H^{\sigma}(\Om_f^\pm)}\le C\|u\|_{H^\sigma(\Om_*^\pm)}.
\eeno

\item[3.] If $u, v\in H^{\sigma}(\Om_*^\pm)$ for $\sigma\in [2,s]$, then
\beno
\|(uv)\circ\Phi_f^{\pm-1}\|_{H^\sigma(\Omega_f^\pm)}\le C\|u\|_{H^\sigma(\Omega_*^\pm)}\|v\|_{H^\sigma(\Omega_*^\pm)}.
\eeno
\end{itemize}
\end{lemma}

\begin{proposition}\label{prop:DNHs}
Let $s\ge 3$ be an integer. If $f\in H^{s+\f12}(\T^2)$, then we have
\beq\label{mhd28}
\|\overline{\mathcal{N}}_f g \|_{H^{s-\f12}}\le C\left(\|f\|_{H^{s+\f12}}\right) \| g \|_{H^{s}(\Om_f^+)}.
\eeq
\end{proposition}
\begin{proof}
We write
\[
\overline{\mathcal{N}}_f g = \left(\widehat{\mathcal{N}}_f^+ g - \mathcal{N}_f^+ g \right) + \left( \mathcal{N}_f^+ g - \mathcal{N}_f^- g \right).
\]
The corresponding estimate for the second term on the right hand side has been shown in the appendix of \cite{SWZ}. The first term in fact satisfies
\beq\label{mhd29}
\|\widehat{\mathcal{N}}_f^+ g - \mathcal{N}_f^+ g \|_{H^{s-\f12}}\le C\left(\|f\|_{H^{s+\f12}}\right) \| g \|_{H^{2}(\Om_f^+)}.
\eeq
To prove this, we first note that
\[
\overline{\mathcal{H}}^+_f g = \widehat{\mathcal{H}}^+_f g - \mathcal{H}^+_f g
\]
satisfies
\[
\Delta\overline{\mathcal{H}}^+_f g  = 0\,\, \text{ in }\Om_f^+, \,\, \overline{\mathcal{H}}^+_f g  = 0\,\, \text{ on }\Gamma_f, \,\, \p_3\overline{\mathcal{H}}^+_f g  = \p_3 g \text{ on }\Gamma^+.
\]
It follows that
\[
\int_{\Om_f^+} \left|\nabla\overline{\mathcal{H}}^+_f g \right|^2 dx = \int_{\Gamma^+} \left(\overline{\mathcal{H}}^+_f g \right) \left( \p_3 g \right) dx' \leq \| \overline{\mathcal{H}}^+_f g \|_{L^{2}} \| \p_3 g \|_{L^{2}}
\]
\[
\leq C \| \overline{\mathcal{H}}^+_f g \|_{H^{1}(\Om_f^+)} \|g\|_{H^2(\Om_f^+)} \leq C \| \nabla\overline{\mathcal{H}}^+_f g \|_{L^{2}(\Om_f^+)} \|g\|_{H^2(\Om_f^+)},
\]
where in the last inequality we applied Poincar\'{e}'s inequality to $\overline{\mathcal{H}}^+_f g$. On the other hand, according to standard interior (and boundary near $\Gamma_f$) elliptic estimates, we have
\[
\|\overline{\mathcal{H}}^+_f g\|_{H^{s}(\widetilde{\Om_f^+})} \leq C\left(\|f\|_{H^{s+\f12}}\right) \| \overline{\mathcal{H}}^+_f g \|_{H^{1}(\Om_f^+)}.
\]
Here $\widetilde{\Om_f^+}$ is any sub-domain of $\Om_f^+$ away from $\Gamma^+$, such as
\[
\widetilde{\Om_f^+} =\big\{x\in\Om_f^+ | f(x')< x_3 < 1- \epsilon_0\big\} \,\,\text{ for sufficiently small }\epsilon_0.
\]
We thus conclude the proof of (\ref{mhd29}), and then (\ref{mhd28}) follows.
\end{proof}

Finally we introduce a commutator estimates that will be used frequently.
\begin{lemma}\label{lem:commutator}
For $s>2$, we have
\beno
\big\|[a, \langle\na\rangle^s]u\big\|_{L^2}\le C\|a\|_{H^{s}}\|u\|_{H^{s-1}}.
\eeno
Here $\langle \na\rangle^s$ is the $s$-order derivatives on $\BT$ defined as follows
\[
\widehat{\langle \na\rangle^s f}(\mathbf{k}) = \left( 1 + |\mathbf{k}|^2 \right)^{\f{s}{2}}\hat{f}(\mathbf{k}), \quad \mathbf{k}=(k_1,k_2), \, k_1,k_2\in\mathbb{Z}.
\]
\end{lemma}

\section{Reformulation of the problem}\label{s3}

In this section, we replace the system (\ref{mhd01})-(\ref{mhd7}) by an equivalent formulation, which consists of the (evolution) equations of the following quantities:
\begin{itemize}
\item the height function of the interface: $f$;
\item the scaled normal velocity on the interface: $\theta=\vu\cdot\vN_f$;
\item the vorticity and current in the fluid region: $\vom=\nabla\times\vu$, $\vj=\nabla\times\vh$;
\item the average of tangential part of velocity and magnetic field on the fixed bottom boundary:
\begin{align*}
\beta_i(t)=\int_{\BT}u_i(t,x',-1)dx',\, \gamma_i(t)=\int_{\BT}h_i(t,x',-1)dx', \, i=1,2.
\end{align*}
\end{itemize}

To simplify notations, from now on we drop the minus superscript "-". Hence,
\[
\Om_f = \Om_f^-, \, \Gamma = \Gamma^-, \, \mathcal{N}_f = \mathcal{N}_f^-, \, \mathcal{H}_f = \mathcal{H}_f^-, \, \text{etc.}.
\]

\subsection{Evolution of the interface and the scaled normal velocity}\label{s31}
Let
\begin{align}
\theta(t,x') = \vu(t,x',f(t,x'))\cdot\vN_f(t,x').
\end{align}
Then we have
\begin{align}\label{eq:form:f}
\p_tf(t,x')=\theta(t,x').
\end{align}
Clearly, $(1+|\nabla f|^2)^{-1/2}\theta$ is the normal component of the fluid velocity on the interface.

According to (\ref{res}), for a vector field $\vv=(v_1,v_2,v_3)$ defined in $\Om^{+}_f$  or $\Om_f$ we calculate $\ul{\vv\cdot\nabla\vv}\cdot\vN_f$ as follows:
\begin{align*}
&(\ul{\vv\cdot\nabla\vv})\cdot\vN_f = \ul{v_1}\ul{\partial_1v_j}N_j + \ul{v_2}\ul{\partial_2 v_j}N_j + \ul{v_3} \ul{\partial_3 v_j}N_j\\
&= \ul{v_1} \p_1\ul{v_j}N_j + \ul{v_2}\p_2\ul{v_j}N_j + \left(\ul\vv \cdot \vN_f\right) \left( \ul{\p_3\vv}\cdot\vN_f \right)\\
&=\ul{v_1}\partial_1(\ul{\vv}\cdot\vN_f)+\ul{v_2}\partial_2(\ul{\vv}\cdot\vN_f)-
\ul{v_1}\ul{v_j}\partial_1N_j-\ul{v_2}\ul{v_j}\partial_2N_j + \left(\ul\vv \cdot \vN_f\right) \left( \ul{\p_3\vv}\cdot\vN_f \right).
\end{align*}
Hereafter we use Einstein's notation of summation for repeated indices $i,j=1,2,3$ as well as summation on $i,j=1,2$ in case of making no confusion. From the calculations above, we obtain the following lemma.

\begin{lemma}\label{rel-uh}
For a vector $\vv=(v_1,v_2,v_3)$ defined in $\Om^{+}_f$ or $\Om_f$,
\begin{align}
&(\ul{{\vv}\cdot{\nabla\vv}})\cdot\vN_f- \left(\ul{\partial_3\vv}\cdot\vN_f\right)\left(\ul{\vv}\cdot\vN_f\right)\nonumber\\
&=\ul{v_1}\partial_1(\ul{\vv}\cdot\vN_f)+\ul{v_2}\partial_2(\ul{\vv}\cdot\vN_f)
+\sum_{i,j=1,2}\ul{v_i}\ul{v_j}\partial_i\partial_jf.
\end{align}
\end{lemma}

By restricting the equation (\ref{mhd01}) to $\Gamma_f$ and taking inner product with $\vN_f$, we deduce from Lemma \ref{rel-uh} (recall $\vh\cdot\vN_f=0$ on $\Gamma_f$) that
\begin{align}
\p_t\theta=&\big(\p_t\vu+\partial_3\vu\p_tf\big)\cdot\vN_f+\vu\cdot\p_t\vN_f\big|_{x_3=f(t,x')}\nonumber\\
=&\big(-\vu\cdot\nabla\vu+\vh\cdot\nabla\vh-\nabla p+\partial_3\vu\p_tf\big)\cdot\vN_f-
\vu\cdot\big(\partial_1\p_tf,\partial_2\p_tf,0\big)\big|_{x_3=f(t,x')} \nonumber \\
=&\big(-(\vu\cdot\nabla)\vu+\partial_3\vu(\vu\cdot\vN_f)\big)\cdot\vN_f+(\vh\cdot\nabla)\vh\cdot\vN_f \nonumber \\
&-\vN_f\cdot\nabla p-\vu\cdot(\partial_1\theta, \partial_2\theta,0)\big|_{x_3=f(t,x')} \nonumber \\
=&-2\big(\ul{u_1}\partial_1\theta+\ul{u_2}\partial_2\theta\big)-\vN_f\cdot\ul{\nabla p}-\sum_{i,j=1,2}
\ul{u_i}\ul{u_j}\partial_i\partial_jf+\sum_{i,j=1,2}\ul{h_i}\ul{h_j}\partial_i\partial_jf.\label{eq:theta-tup}
\end{align}

To give the trace of the pressure $p$ on $\Gamma_f$, we first take divergence to (\ref{mhd01}) to yield
\begin{align}
\Delta p=\mathrm{tr}(\nabla\vh\nabla\vh)-\mathrm{tr}(\nabla\vu\nabla\vu).
\end{align}
Let $p_{\vv_1, \vv_2}$ be the solution of the following mixed boundary value problem:
\beq\label{mhd39}
\left\{
\begin{array}{l}
\Delta p_{\vv_1, \vv_2}= -\mathrm{tr}(\nabla\vv_1\nabla\vv_2)
\text{ in } \Omega_f,\\
p_{\vv_1, \vv_2}=0 \text{ on } \Gamma_f,\,\,
\ve_3\cdot\nabla p_{\vv_1, \vv_2}=0 \text{ on } \Gamma.
\end{array}\right.
\eeq
Since $\ul{p}=p|_{\Gamma_f}=\frac{1}{2}|\hat{\vh}|^2|_{\Gamma_f}$, we obtain the following representation formula for the pressure $p$:
\begin{align}
p=\mathcal{H}_f\ul{p} + p_{\vu, \vu}-p_{\vh, \vh} = \frac{1}{2}\mathcal{H}_f |\ul{\hat{\vh}}|^2+p_{\vu, \vu}-p_{\vh, \vh}.
\end{align}
It follows from (\ref{eq:theta-tup}) that
\begin{align}
\p_t\theta =& -2(\ul{u_1}\partial_1\theta+\ul{u_2}\partial_2\theta\big)-\sum_{i,j=1,2}
\big(\ul{u_i}\ul{u_j}-\ul{h_i}\ul{h_j}\big)\partial_i\partial_jf\nonumber\\
&-\frac12\mathcal{N}_f|\ul{\hat{\vh}}|^2
-\vN_f\cdot\ul{\nabla(p_{\vu,\vu}-p_{\vh,\vh})}.\label{mhd31}
\end{align}
Note that
\begin{align}\label{mhd32}
-\mathcal{N}_f|\ul{\hat{\vh}}|^2 =&
\big(\widehat{\mathcal{N}}_f^+-\mathcal{N}_f\big)|\hat{\vh}|^2 -\vN_f\cdot\ul{\nabla(|\hha|^2-\widehat{\mathcal{H}}^+_f |\hha|^2)}+\vN_f\cdot\ul{\nabla(|\hha|^2)}\\
=& \overline{\mathcal{N}}_f |\hat{\vh}|^2-\vN_f\cdot\ul{\nabla(|\hha|^2-\hat{\mathcal{H}}^+_f |\hha|^2)}+\vN_f\cdot\ul{\nabla(|\hha|^2)}. \nonumber
\end{align}
Furthermore, by Lemma \ref{rel-uh},
\begin{align}\label{mhd33}
\sum_{i,j=1,2}\ul{\hat h_i}\ul{\hat h_j}\partial_i\partial_jf
=(\ul{\hha\cdot\nabla\hha})\cdot\vN_f
=\ha_i\partial_i\ha_jN_j=\ha_i\partial_j\ha_iN_j=\frac12\vN_f\cdot\nabla |\hha|^2,
\end{align}
where we used $\curl\hha = 0$. From (\ref{mhd31})-(\ref{mhd33}), we finally obtain
\begin{align}\nonumber
\p_t\theta=&-
2(\ul{u_1}\partial_1\theta+\ul{u_2}\partial_2\theta\big)-\sum_{i,j=1,2}
\big(\ul{u_i}\ul{u_j}-\ul{h_i}\ul{h_j}
-\ul{\hat h_i}\ul{\hat h_j}\big)\partial_i\partial_jf\\
&-\vN_f\cdot\ul{\nabla(p_{\vu,\vu}-p_{\vh,\vh})}-\frac12\vN_f\cdot\ul{\nabla(|\hha|^2-\hat{\mathcal{H}}^+_f |\hha|^2)}
+\frac12\overline{\mathcal{N}}_f |\hha|^2 .\label{eq:theta}
\end{align}

\subsection{The equations for the vorticity and current}
Let
\begin{align*}
\vom = \nabla\times\vu,\quad \vj = \nabla\times\vh
\end{align*}
be the vorticity and current in $\Om_f$ respectively. Then $\vom, \vj$ satisfy
\begin{align}\label{eq:vorticity-w}
&\p_t\vom+\vu\cdot\nabla\vom-\vh\cdot\nabla\vj=\vom\cdot\nabla\vu-\vj\cdot\nabla\vh \text{ in }Q_T,\\
&\p_t\vj+\vu\cdot\nabla\vj-\vh\cdot\nabla\vom=\vj\cdot\nabla\vu
-\vom\cdot\nabla\vh-2\nabla u_i\times\nabla h_i \text{ in }Q_T. \label{eq:vorticity-h}
\end{align}
Here we used the fact that
\begin{align*}
&\veps^{ijk}\partial_ju_l\partial_lh_k-\veps^{ijk}\partial_jh_l\partial_lu_k\\
&=\veps^{ijk}\partial_ju_l(\partial_lh_k-\partial_kh_l)+\veps^{ijk}\partial_ju_l\partial_kh_l-\veps^{ijk}\partial_jh_l\partial_ku_l
+\veps^{ijk}\partial_jh_l(\partial_ku_l-\partial_lu_k)\\
&=-\vj\cdot\nabla\vu+\vom\cdot\nabla\vh-2\nabla u_i\times\nabla h_i.
\end{align*}

As in \cite{SWZ}, to uniquely recover a divergence free vector field from its curl and normal component(on the bottom boundary) in $\Omega_f$,
we need to prescribe the mean value of its tangential components on the bottom boundary.
To this end, let
\begin{align*}
\beta_i(t)=\int_{\BT}u_i(t,x',-1)dx',\, \gamma_i(t)=\int_{\BT}h_i(t,x',-1)dx', \, i=1,2.
\end{align*}
Thanks to $u_3(t,x',-1)\equiv 0$, we deduce that for $i=1,2$,
\begin{align*}
  \p_tu_i+u_{j}\partial_{j}u_i-h_{j}\partial_{j}h_i-\partial_ip=0\, \text{on}\, \Gamma,
\end{align*}
which yields
\begin{align*}
\p_t\beta_i+\int_{\Gamma}\big(u_{j}\partial_{j}u_i-h_{j}\partial_{j}h_i\big)dx'=0,
\end{align*}
or equivalently
\begin{align}
\beta_i(t)=\beta_i(0)-\int_0^t\int_{\Gamma}\big(u_{j}\partial_{j}u_i-h_{j}\partial_{j}h_i\big)dx'dt.
\end{align}
Similarly, we have
\begin{align}
\gamma_i(t)=\gamma_i(0)-\int_0^t\int_{\Gamma}\big(u_{j}\partial_{j}h_i-h_{j}\partial_{j}u_i\big)dx'dt.
\end{align}

Finally, to solve $\hat{\vh}$, and to recover $\vu,\vh$ from $\vom,\vj$ in $\Omega_f$, one needs to solve two types of div-curl system. We leave it to the next section.

\section{Div-curl system}\label{s4}

In this section, we consider two div-curl systems, which have also been considered in \cite{CS1} for the bounded domain.
Assume that $\Gamma_f$ is a given graph with $f\in H^{s+\f12}(\T^2)$ for $s\ge 2$
satisfying
\[
-(1-c_0)\le f\le (1-c_0).
\]
The first system reads as follows
\beq\label{eq:div-curl}
\left\{\begin{aligned}
&\div\vv=g,\quad \curl \vv=\vom\, \text{ in }\,  \Om_f,\\
&\vv\cdot\vN_f = \vartheta \text{ on } \Gamma_{f}, \,\, \vv\cdot\ve_3 = 0, \,\, \int_{\Gamma} v_i dx'=\alpha_i \text{ on } \Gamma,\, i=1,2,
\end{aligned}\right.
\eeq
where $\vom$ and $\vartheta$ are given functions in $\Om_f$ and $\Gamma_f$ respectively, and $\alpha_i, i=1,2$ are given real numbers. Assume that $\vom,\theta$ satisfy the following compatibility conditions:
\beq\label{mhd41}
\div\vom=0 \,\text{ in }\, \Omega_f, \quad \int_{\Gamma}\om_3 dx'=0,\quad
 \int_{\Om_f} g dx= \int_{\T^2} \vartheta dx'.
\eeq
The following proposition has been proved in \cite{SWZ}.
\begin{proposition}\label{prop:div-curl}
Let $f\in H^{s+\f12},\,s\ge 2$ and $\sigma \in [1,s]$. Assume $g, \vom \in H^{\sigma-1}(\Om_f)$, $\vartheta\in H^{\sigma-\frac12}(\Gamma_f)$
satisfying (\ref{mhd41}). Then there exists a unique $\vv\in H^{\sigma}(\Om_f)$ of (\ref{eq:div-curl}) so that
\begin{align*}
\|\vv\|_{H^{\sigma}(\Om_f)}\le C\left(\|f\|_{H^{s+\f12}}\right)\left( \|g\|_{H^{\sigma-1}(\Om_f)}+
\|\vom\|_{H^{\sigma-1}(\Om_f)}+\|\vartheta\|_{H^{\sigma-\frac12}(\Gamma_f)}+|\alpha_1|+|\alpha_2|\right).
\end{align*}
\end{proposition}

The second div-curl system is
\beq\label{eq:div-curl-2}\left\{\begin{aligned}
&\div \hat\vh=\hat g,\quad\curl \hat\vh = \hat\vom\text{ in } \Om_f^+,\\
&\hat\vh\cdot\vN_f = \hat\vartheta\, \text{ on }\, \Gamma_f,   \,\,
\hat\vh\times\ve_3 = \hat{\mathbf{J}} \text{ on } \Gamma^{+}.
\end{aligned}\right.\eeq
Here $\hat{\mathbf{J}}=(\hat{J}_1(x'),\hat{J}_2(x'),\hat{J}_3(x'))$ is a given vector on $\Gamma^+$.
To solve this boundary value problem, we need the following compatibility conditions on $\hat\vom$ and $\hat{\mathbf{J}}$:
\beq\label{mhd42}
\div\hat\vom=0\,\text{ in }\, \Omega_f^+, \quad \p_1\hat{J}_1 + \p_2\hat{J}_2 = \hat\om_3,\quad \hat{J}_3 = 0\, \text{ on }\,\Gamma^+.
\eeq
\begin{proposition}\label{prop:div-curl-2}
Let $f\in H^{s+\f12},\,s\ge 2$ and $\sigma \in [1,s]$. Assume $\hat{g}, \hat\vom \in H^{\sigma-1}(\Om_f^+)$,
$\hat\vartheta, \hat{\mathbf{J}}\in H^{\sigma-\f12}$ satisfying the compatibility condition (\ref{mhd42}).
Then there exists a unique $\hat\vh\in H^{\sigma}(\Om_f^+)$ of the div-curl system (\ref{eq:div-curl-2}) so that
\beq\label{mhd43}
\|\hat\vh\|_{H^{\sigma}(\Omega_f^+)}\le C\left(\|f\|_{H^{s+\f12}}\right)\left(\|\hat{g}\|_{H^{\sigma-1}(\Om_f^+)}+
 \|\hat\vom\|_{H^{\sigma-1}(\Om_f^+)}+\|\hat\vartheta\|_{H^{\sigma-\f12}}+\|\hat{\mathbf{J}}\|_{H^{\sigma-\f12}}\right).
\eeq
\end{proposition}
\begin{proof}
By Proposition \ref{prop:div-curl}, there is a vector field $\tilde{\vh}$ satisfies:
\begin{align*}\left\{\begin{aligned}
&\div \tilde\vh=0,\quad \curl \tilde\vh = \hat\vom\,\text{ in }\,\Om_f^+,\\
&\tilde\vh\cdot\vN_f = 0\,\text{ on } \Gamma_f,   \,\,
\tilde\vh\cdot\ve_3 = 0\, \text{ on } \Gamma^{+},
\end{aligned}\right.
\end{align*}
and $\|\tilde\vh\|_{H^{\sigma}(\Omega_f^+)}\le C\left(\|f\|_{H^{s+\f12}}\right)\|\hat\vom\|_{H^{\sigma-1}(\Om_f^+)}.$
On the other hand, according to (\ref{mhd42}), $\partial_1(\hat{J}_1-\tilde h_2)+\partial_2(\hat{J}_2+\tilde h_1)=\partial_1\hat{J}_1+\partial_2\hat{J}_2-\hat\om_3=0$, thus, there exists a scalar function $\hat{j}\in H^{\sigma+\f12}$ such that
\[
\hat{J}_1-\tilde h_2 = \p_2\hat{j}, \quad \hat{J}_2+\tilde h_1 = - \p_1\hat{j}.
\]
Let $\tilde{j}\in H^{\sigma+1}(\Om_f^+)$ be an extension of $\hat{j}$ to $\Om_f^+$ such that
\[
\tilde{j}=0 \text{ near }\Gamma_f,\quad \|\tilde{j}\|_{H^{\sigma+1}(\Om_f^+)}\le C\left(\|f\|_{H^{s+\f12}}\right)
\Big(\|\hat\vom\|_{H^{\sigma-1}(\Om_f^+)}+\|\hat{\mathbf{J}}\|_{H^{\sigma-\f12}}\Big).
\]

We consider the following mixed Dirichlet-Neumman problem:
\beq\label{mhd45}\left\{\begin{aligned}
&\Delta\phi = \Delta\tilde{j}-\hat{g}\quad \text{ in }\, \Om_f^+,\\
&\nabla\phi\cdot\vN_f = -\hat\vartheta\, \text{ on }\, \Gamma_f,   \,\,
\phi=0 \,\text{ on }\, \Gamma^{+}.
\end{aligned}\right.\eeq
The existence, uniqueness and regularity of this mixed boundary value problem are standard \cite{Lan}. Moreover,
\beno
\|\phi\|_{H^{\sigma+1}(\Om_f^+)}\le C\left(\|f\|_{H^{s+\f12}}\right)\left(\|\hat{g}\|_{H^{\sigma-1}(\Om_f^+)}+
 \|\hat\vom\|_{H^{\sigma-1}(\Om_f^+)}+\|\hat\vartheta\|_{H^{\sigma-\f12}}+\|\hat{\mathbf{J}}\|_{H^{\sigma-\f12}}\right).
\eeno
Let
\[
\hat{\vh} = \tilde\vh+\nabla\tilde{j} - \nabla\phi.
\]
Obviously, $\div\hat{\vh}=  \Delta\tilde{j}- \Delta\phi=\hat{g}$, $\curl\hat{\vh}=\curl\tilde\vh=\hat\vom$ in $\Om_f^+$. Moreover,
\[
\hat{\vh}\cdot\vN_f = \tilde\vh\cdot\vN_f+ \nabla\tilde{j}\cdot\vN_f - \nabla\phi\cdot\vN_f = \hat\vartheta
\]
due to $\tilde{j}=0$ near $\Gamma_f$. As to the boundary value on $\Gamma^+$, we note that since $\p_1\phi,\p_2\phi = 0$ on $\Gamma^+$,
\[
\hat{\vh}|_{\Gamma^+} \times \ve_3 = \left( \tilde\vh+\nabla\tilde{j} - \nabla\phi\right)|_{\Gamma^+} \times \ve_3
= \left(\tilde h_2+\p_2\hat{j}-\p_2\phi, -\tilde h_1-\p_1\hat{j} + \p_1\phi, 0 \right) = (\hat{J}_1, \hat{J}_2, 0).
\]
We conclude the proof of existence of solution to the system (\ref{eq:div-curl-2}) and the regularity estimate (\ref{mhd43}).

The proof of the uniqueness is similar to the proof of Lemma 5.4 in \cite{SWZ}. We present it here for completeness. It suffices to consider $\hat{g},\hat{\vom}, \hat\vartheta=0$ and $\hat{\mathbf{J}}=0$.
We periodically extend $\Omega_{f}^+$ to be a unbounded domain in $\mathbb{R}^3$, which is denoted by $\Omega_{f,p}^+$.
Then $\hat{\vh}=\na\phi$, where $\phi$ is a function on $\Omega_{f,p}^+$.  Let $\zeta(x)=\phi(x_1+2\pi,x_2,x_3)-\phi(x_1,x_2,x_3)$ for $x\in \Omega_{f,p}^+$. Then $\nabla \zeta(x)=0$,
and thus $\zeta(x)$ is a constant. The condition $\partial_1\phi=\hat h_1=0$ on $\Gamma^{+}$ implies $\zeta(x)\equiv 0$, and then $\phi$ is periodic in $x_1$ in $\Omega_{f,p}^+$.
Similarly, $\phi$ is periodic in $x_2$. Thus $\phi$ can be viewed as a function on $\Omega_{f}^+$ and is a constant on $\Gamma^+$.
Then we can obtain a function $\phi$ in $\Omega_{f}^+$ satisfies
\beno
\Delta\phi=0\quad \text{in}\,\,\Omega_f^+,\quad \vN_f\cdot\na \phi=0\quad \text{on } \Gamma_f,  \quad \phi=\text{constant}\quad \text{on } \Gamma^{+}.
\eeno
Thus, the uniqueness of (\ref{mhd45}) implies $\phi\equiv$ constant, and then $\hat{\vh}\equiv0$.
\end{proof}

\section{Uniform estimates for the linearized system}\label{s5}
Given ${f}(t,x'),{\vu}(t,x),{\vh}(t,x),\hat\vh(t,x)$, we assume there exist positive constants $\delta_0,c_0$ and $L_0,L_1,L_2$ such that for $t\in [0,T]$,
\begin{itemize}\label{ass:stability}
\item $\|(\vu, \vh, \hat\vh)(t)\|_{L^{\infty}(\Gamma_f)}\le L_0$;
\item $\|f(t)\|_{H^{s+\f12}}+\|\pa_t f(t)\|_{H^{s-\f12}}+\|\big(\vu ,\vh\big)(t)\|_{H^{s}(\Om_f)}
  +\|\hat\vh(t)\|_{H^{s}(\Omega_f^+)} + \|\p_t\hat\vh(t)\|_{H^{s-1}(\Omega_f^+)} \le L_1$;
\item $\|(\p_t\vu, \p_t\vh\|_{L^{\infty}(\Gamma_f)}\le L_2$;
\item $\|f(t)-f_*\|_{H^{s-\f12}}\le \delta_0,\, -(1-c_0)\le f(t,x')\le (1-c_0),x'\in\BT$;
\item $\Lambda(\vh, \hat\vh)(t)\ge c_1$;
\end{itemize}
together with
\ben\label{ass:boun}
\left\{
\begin{array}{l}
\div\vu=\div\vh=0 \text{ in } \Om_f,\\
{\vh}\cdot\vN_f=0, \, \p_t f = {\vu}\cdot\vN_f \text{ on }\Gamma_f,\\
u_3=h_3=0 \text{ on } \Gamma,
\end{array}\right.\quad\left\{
\begin{array}{l}
\div\hat\vh=0, \, \curl\hat\vh = 0 \text{ in } \Omega_f^+,\\
{\hat\vh}\cdot\vN_f=0 \text{ on }\Gamma_f,\\
\hat{\vh}\times \ve_3 = \hat{\mathbf{J}}\text{ on }\Gamma^+.
\end{array}\right.
\een

In this section, we linearize the equivalent system derived in Section \ref{s3} around $(f,\vu,\vh,\hat\vh)$, and present the uniform energy estimates for the linearized system. First of all, we give the following lemma on a new formulation of the stability condition (\ref{stab1}).
\begin{lemma}\label{lem:stability}
Under the stability condition (\ref{stab1}), there exists $c_1>0$ such that
\beq\label{stab2}
\Lambda(\vh,\hat\vh) \eqdefa\inf_{x\in\Gamma_t}\inf_{\varphi_1^2+\varphi_2^2=1}(h_1\varphi_1+h_2\varphi_2)^2+(\hat{h}_1\varphi_1+\hat{h}_2\varphi_2)^2\ge c_1.
\eeq
\end{lemma}
\begin{proof}
Let $\vq=(q_1, q_2, q_3)\bot\vN_f$ with $q_3=q_1\partial_1f+q_2\partial_2f$ and $(q_1,q_2)$ determined by
\begin{equation*}
\left(
  \begin{array}{cc}
    1+(\partial_1f)^2 & \partial_1f\partial_2f \\
    \partial_1f\partial_2f & 1+(\partial_2f)^2  \\
  \end{array}
\right)
\left(
         \begin{array}{c}
           q_1 \\
           q_2 \\
         \end{array}
\right)
= \left(
         \begin{array}{c}
           \varphi_1 \\
           \varphi_2 \\
         \end{array}
\right).
\end{equation*}
Then by the fact $\vh\cdot\vN_f=\hat\vh\cdot\vN_f=0$, we get
\beno
h_1\varphi_1+h_2\varphi_2=\sum_{i=1}^3h_i q_i,\,
\hat{h}_1\varphi_1+\hat{h}_2\varphi_2=\sum_{i=1}^3\hat{h}_i q_i,
\eeno
which along with (\ref{stab1}) gives
\beno
\inf_{\varphi_1^2+\varphi_2^2=1} (h_1\varphi_1+h_2\varphi_2)^2+(\hat{h}_1\varphi_1+\hat{h}_2\varphi_2)^2>0.
\eeno
Since the inequality above holds for all $x\in\Gamma_f$, there exists a constant $c_1>0$ such that
\beno
\inf_{\varphi_1^2+\varphi_2^2=1} (h_1\varphi_1+h_2\varphi_2)^2+(\hat{h}_1\varphi_1+\hat{h}_2\varphi_2)^2\ge c_1,
\eeno
which yields (\ref{stab2}).
\end{proof}

For the system (\ref{eq:form:f}) and (\ref{eq:theta}), we introduce the following linearized system:
\begin{equation}\label{sys:linear-H}
\left\{\begin{aligned}
\p_t\bar{f}=&~\bar{\theta},\\
\p_t\bar{\theta}=&-
2(\ul{u_1}\partial_1\bar\theta+\ul{u_2}\partial_2\bar\theta\big)+\sum_{i,j=1,2}
\big(-\ul{u_i}\ul{u_j}+\ul{h_i}\ul{h_j}+\ul{\hat h_i}\ul{\hat h_j}\big)\partial_i\partial_j\bar{f}+\mathfrak{g},
\end{aligned}\right.
\end{equation}
where
\begin{align}
\mathfrak{g}=&-\vN_f\cdot\ul{\nabla(p_{\vu,\vu}-p_{\vh,\vh})}-\frac12\vN_f\cdot\ul{\nabla(|\hha|^2-\widehat{\mathcal{H}}^+_f |\hha|^2)}+\frac12\overline{\mathcal{N}}_f|\hat{\vh}|^2.\label{eq:g-def}
\end{align}
We remark that $\int_{\BT}\bar\theta dx'$ may not vanish since we have performed a linearization.\medskip

Now we introduce the energy functional $E_s$ defined by
\beno
E_s(t) = \big\|(\p_t+\ul{u_i}\partial_i)\Ds\bar{f}\big\|_{L^2}^2
+\frac12\left\|\ul{h_i}\partial_i\Ds\bar{f}\right\|_{L^2}^2
+\frac12\left\|\ul{\hat{h_i}}\partial_i\Ds\bar{f}\right\|_{L^2}^2.
\eeno
Also we define the standard energy
\[
\mathcal{E}_s(t) = \|\bar{f}(t)\|_{H^{s+\f12}}^2+\|\p_t\bar{f}(t)\|_{H^{s-\f12}}^2.
\]
It is easy to see that there exists $C(L_0)>0$ so that
\begin{align}\label{linear:equi-norm-1}
E_s(t)\le C(L_0) \mathcal{E}_s(t).
\end{align}
The stability condition  guarantees that there exists $C(c_1,L_0)$ so that
\begin{align}\label{linear:equi-norm-2}
\mathcal{E}_s(t)
\le C(c_0,L_0)\Big\{E_s(t)+\|\p_t\bar{f}\|_{L^2}^2+\|\bar{f}\|_{L^2}^2\Big\}.
\end{align}

Before to state the energy estimates, we first give the following lemma concerning $\mathfrak{g}$ defined by (\ref{eq:g-def}).
\begin{lemma}\label{lem:non-g}
It holds that
\begin{align*}
\|\mathfrak{g}\|_{H^{s-\f12}}\le C(L_1).
\end{align*}
\end{lemma}
\begin{proof}
According to the definition of $p_{\vu,\vu},p_{\vh,\vh}$ (see (\ref{mhd39})), we obtain by standard elliptic estimates that
\begin{align*}
\| \vN_f\cdot\ul{\nabla(p_{\vu,\vu}-p_{\vh,\vh})} \|_{H^{s-\f12}}
\le& C(L_1)\|\ul{\nabla(p_{\vu,\vu}-p_{\vh,\vh})}\|_{H^{s-\f12}}
\le C(L_1)\big\|\nabla(p_{\vu,\vu}, p_{\vh,\vh})\big\|_{H^{s}(\Om_f)}\\
\le& C(L_1)\|(\vu,\vh)\|_{H^s(\Om_f)}\le C(L_1).
\end{align*}
Applying similar argument to $\hat{p} = |\hha|^2-\widehat{\mathcal{H}}^+_f |\hha|^2$ yields the same estimate for the second term in $\mathfrak{g}$. Finally, the same estimate for the third term follows from (\ref{mhd28}) in Proposition \ref{prop:DNHs}.\end{proof}

\begin{proposition}\label{prop:f-L}
Given initial data $\bar f_0\in H^{s+\f12}$, $\bar \theta_0\in H^{s-\f12}$, there exists a unique solution $(\bar{f},\bar{\theta})\in C\big([0,T];H^{s+\f12}\times H^{s-\f12}\big)$ to the system (\ref{sys:linear-H}) so that
\beq\label{mhd51}
\sup_{t\in[0,T]}\mathcal{E}_s(t) \le C(c_1,L_0)\left( 1+\|\bar\theta_0\|_{H^{s-\f12}}^2 + \|\bar f_0\|_{H^{s+\f12}}^2\right)e^{C(c_1, L_1,L_2)T}.
\eeq
\end{proposition}

\begin{proof}
We only present the uniform estimates, which ensure the existence and uniqueness of the solution.
Using the fact that
\begin{align*}
\pa_t^2\bar f=-2\sum_{i=1,2}\ul{u_i}\pa_i\pa_t\bar f+\sum_{i,j=1,2}(-\ul{u_i}\ul{u_j}+\ul{h_i}\ul{h_j}+\ul{\hat{h_i}}\ul{\hat{h}_j})\partial_i\partial_j\bar f+\mathfrak{g},
\end{align*}
a direct calculation shows that
\begin{align*}
&\frac12\frac{d}{dt}\left\|(\p_t+\ul{u_i}\partial_i)\Ds\bar{f}\right\|_{L^2(\BT)}^2\\
&=\Big\langle(\p_t+\ul{u_i}\partial_i)\Ds\bar{f}, \Ds\partial_{t}^2\bar{f}+\ul{u_i}\partial_i(\Ds\p_t\bar{f})
+\p_t\ul{u_i}\partial_i\Ds\bar{f}\Big\rangle\\
&=\Big\langle(\p_t+\ul{u_i}\partial_i)\Ds\bar{f}, \Ds\big(-2\ul{u_i}\partial_i\p_t\bar{f}
-\ul{u}_i\ul{u_j}\partial_i\partial_j\bar{f}+(\ul{h_i}\ul{h_j}+\ul{\hat{h_i}}\ul{\hat{h_j}})\partial_i\partial_j\bar{f}\big)\Big\rangle\\
\quad&+\Big\langle(\p_t+\ul{u_i}\partial_i)\Ds\bar{f}, \Ds
\mathfrak{g}+\ul{u_i}\partial_i(\Ds\p_t\bar{f})
+\p_t\ul{u_i}\partial_i\Ds\bar{f}\Big\rangle\\
&=\Big\langle(\p_t+\ul{u_i}\partial_i)\Ds\bar{f}, -\ul{u_i}\partial_i\Ds\p_t\bar{f}\Big\rangle\\
&\quad+\Big\langle(\p_t+\ul{u_i}\partial_i)\Ds\bar{f},-\ul{u_i}\ul{u_j}\partial_i\partial_j\Ds\bar{f}
+(\ul{h_i}\ul{h_j}+\ul{\hat{h_i}}\ul{\hat{h_j}})
\partial_i\partial_j\Ds\bar{f})\Big\rangle\\
&\quad+2\Big\langle(\p_t+\ul{u_i}\partial_i)\Ds\bar{f},\big[\ul{u_i}, \Ds\big]\partial_i\p_t\bar{f}\Big\rangle\\
&\quad+\Big\langle(\p_t+\ul{u_i}\partial_i)\Ds\bar{f},\big[\ul{u_i}\ul{u_j}-\ul{h_i}\ul{h_j}
-\ul{\hat{h_i}}\ul{\hat{h}_j},\Ds\big]\partial_i\partial_j\bar{f})\Big\rangle\\
&\quad+\Big\langle(\p_t+\ul{u_i}\partial_i)\Ds\bar{f},\Ds
\mathfrak{g}+\p_t\ul{u_i}\partial_i\Ds\bar{f}\Big\rangle \triangleq I_1 + \cdots + I_5.
\end{align*}
It follows from Lemma \ref{lem:commutator} that
\begin{align*}
I_3\le&2\|(\p_t+\ul{u_i}\partial_i)\Ds\bar{f}\|_{L^2}\big\|\big[\ul{u_i}, \Ds\big]\partial_i\p_t\bar{f}\big\|_{L^2}\\
\le& CE_s(t)^\f12\|\ul{\vu}\|_{H^{s-\f12}}\|\pa_t\bar f\|_{H^{s-\f12}},
\end{align*}
as well as
\beno
I_4\le CE_s(t)^\f12\Big(\|\ul{\vu}\|_{H^{s-\f12}}^2+\|\ul{\vh}\|_{H^{s-\f12}}^2
+\|\ul{\hat{\vh}}\|_{H^{s-\f12}}^2\Big)\|\bar f\|_{H^{s+\f12}}.
\eeno
Also we have
\beno
I_5\le E_s(t)^\f12\Big(\|\mathfrak{g}\|_{H^{s-\f12}}+\|\p_t\vu\|_{L^\infty}\|\bar f\|_{H^{s+\f12}}\Big).
\eeno

We get by integration by parts that
\begin{align*}
&\Big\langle\p_t\Ds\bar{f}, ~-\ul{u_i}\partial_i\Ds\p_t\bar{f}\Big\rangle
\le \|\partial_i\ul{u_i}\|_{L^\infty}\|\p_t\bar{f}\|_{H^{s-\f12}}^2,\\
&\Big\langle \ul{u_i}\partial_i\Ds\bar{f}, ~-\ul{u_i}\partial_i\Ds\p_t\bar{f}\Big\rangle
+\frac12\frac{d}{dt}\|\ul{u_i}\partial_i\Ds\bar{f}\|_{L^2}^2\\
&=\Big\langle \ul{u_i}\partial_i \Ds\bar{f},\p_t\ul{u_i}\partial_i\Ds\bar{f}\Big\rangle\le
\|\vu\|_{L^\infty} \|\p_t\vu\|_{L^\infty}\|\bar{f}\|_{H^{s+\f12}}^2,
\end{align*}
which give rise to
\beno
I_1\le -\frac12\frac{d}{dt}\|\ul{u_i}\partial_i\Ds\bar{f}\|_{L^2}^2+\big(1+\|\vu\|_{W^{1,\infty}}+\|\p_t\vu\|_{L^\infty}\big)^2
\Big(\|\bar{f}\|_{H^{s+\f12}}^2+\|\p_t\bar{f}\|_{H^{s-\f12}}^2\Big).
\eeno
Similarly, we have
\begin{align*}
&\Big\langle\p_t\Ds\bar{f}, -\ul{u_i}\ul{u_j}\partial_i\partial_j\Ds\bar{f}\Big\rangle
-\frac12\frac{d}{dt}\|\ul{u_i}\partial_i\Ds\bar{f}\|_{L^2}^2\\
&=-\Big\langle \ul{u_i}\partial_i\Ds\bar{f}, ~\p_t\ul{u_i}
\partial_i\Ds\bar{f}\Big\rangle+\Big\langle\Ds\p_t\bar{f}, ~\partial_i(\ul{u_i}\ul{u_j})\partial_j\Ds\bar{f}\Big\rangle\\
&\le\|\vu\|_{L^\infty}\big(\|\p_t\vu\|_{L^\infty}+\|\nabla \vu\|_{L^\infty}\big)
\Big(\|\bar{f}\|_{H^{s+\f12}}^2+\|\p_t\bar{f}\|_{H^{s-\f12}}^2\Big),
\end{align*}
and
\begin{align*}
&\Big\langle \ul{u_k}\partial_k\Ds\bar{f}, -\ul{u_i}\ul{u_j}\partial_i\partial_j\Ds\bar{f}\Big\rangle\\
&=\Big\langle \partial_i(\ul{u_k}\ul{u_i}\ul{u_j})\partial_k\Ds\bar{f}, \partial_j\Ds\bar{f}\Big\rangle
-\Big\langle \ul{u_k}\ul{u_i}\ul{u_j}\partial_k\partial_i\Ds\bar{f},\partial_j\Ds\bar{f}\Big\rangle\\
&=\Big\langle \partial_i(\ul{u_k}\ul{u_i}\ul{u_j})\partial_k\Ds\bar{f}, \partial_j\Ds\bar{f}\Big\rangle
-\Big\langle \ul{u_k}\partial_k(\ul{u_i}\partial_i\Ds\bar{f}),\ul{u_j}\partial_j\Ds\bar{f}\Big\rangle\\
&\quad+\Big\langle \ul{u_k}(\partial_k\ul{u_i})\partial_i\Ds\bar{f}), \ul{u_j}\partial_j\Ds\bar{f}\Big\rangle
\le C \|\vu\|^2_{L^\infty}\|\nabla \vu\|_{L^\infty}\|\bar{f}\|_{H^{s+\f12}}^2,
\end{align*}
as well as
\begin{align*}
&\Big\langle\p_t\Ds\bar{f}, \ul{h_i}\ul{h_j}\partial_i\partial_j\Ds\bar{f}\Big\rangle\\
&\le-\frac12\frac{d}{dt}\|\ul{h_i}\partial_i\Ds\bar{f}\|_{L^2}^2+\|\vh\|_{L^\infty}
\big(\|\p_t\vh\|_{L^\infty}+\|\nabla\vh\|_{L^\infty}\big)
\Big(\|\bar{f}\|_{H^{s+\f12}}^2+\|\p_t\bar{f}\|_{H^{s-\f12}}^2\Big),\\
&\Big\langle\p_t\Ds\bar{f}, \ul{\hat{h_i}}\ul{\hat{h_j}}\partial_i\partial_j\Ds\bar{f}\Big\rangle\\
&\le-\frac12\frac{d}{dt}\|\ul{\hat{h_i}}\partial_i\Ds\bar{f}\|_{L^2}^2+\| {\hat\vh}\|_{L^\infty}
\big(\|\p_t {\hat\vh}\|_{L^\infty}+\|\nabla  {\hat\vh}\|_{L^\infty}\big)
\Big(\|\bar{f}\|_{H^{s+\f12}}^2+\|\p_t\bar{f}\|_{H^{s-\f12}}^2\Big).
\end{align*}
Thus, we obtain
\begin{align*}
I_2\le& \frac12\frac{d}{dt}\|\ul{u_i}\partial_i\Ds\bar{f}\|_{L^2}^2
-\frac14\frac{d}{dt}\|\ul{h_i}\partial_i\Ds\bar{f}\|_{L^2}^2-\frac14\frac{d}{dt}\|\ul{\hat{h_i}}\partial_i\Ds\bar{f}\|_{L^2}^2\\
&+C\big(1+\|( {\vu}, {\vh},  {\hat\vh})\|_{W^{1,\infty}}+\|( \p_t{\vu}, \p_t{\vh}, \p_t{\hat\vh})\|_{L^\infty}\big)^3\Big(\|\bar{f}\|_{H^{s+\f12}}^2+\|\p_t\bar{f}\|_{H^{s-\f12}}^2\Big).
\end{align*}

Collecting these estimates of $I_1,\cdots, I_5$ above, we conclude that
\begin{align*}
\frac{d}{dt}E_s(t)\le \|\mathfrak{g}\|_{H^{s-\f12}}^2 + C(L_0)\big(1+\|(\ul{\vu},\ul{\vh}, \ul{\hat\vh})\|_{H^{s-\f12}}+\|(\p_t{\vu}, \p_t{\vh},\p_t{\hat\vh})\|_{L^\infty}\big)^3\mathcal{E}_s(t).
\end{align*}
On the other hand, it is obvious that
\begin{align*}
\frac{d}{dt}\big(\|\p_t\bar{f}\|_{L^2}^2+\|\bar{f}\|_{L^2}^2\big)\le C(L_0)\mathcal{E}_s(t)
+\|\mathfrak{g}\|_{L^2}^2.
\end{align*}
Then by (\ref{linear:equi-norm-2}), we deduce that
\begin{align*}
\mathcal{E}_s(t)\le & C(c_1,L_0)\Big(\|\bar\theta_0\|_{H^{s-\f12}}^2+\|\bar f_0\|_{H^{s+\f12}}^2+\int_0^t\|\mathfrak{g}(\tau)\|_{H^{s-\f12}}^2d\tau\\
&+\int_0^t\big(1+\|(\ul{\vu},\ul{\vh},\ul{\hat\vh})(\tau)\|_{H^{s-\f12}}+\|\p_t( {\vu}, {\vh}, {\hat\vh})(\tau)\|_{L^\infty}\big)^3\mathcal{E}_s(\tau)d\tau\Big),
\end{align*}
which together with Lemma \ref{lem:basic} gives rise to
\begin{align*}
\mathcal{E}_s(t)\le C(c_1,L_0)\Big(&\|\bar\theta_0\|_{H^{s-\f12}}^2+\|\bar f_0\|_{H^{s+\f12}}^2+\int_0^t\|\mathfrak{g}(\tau)\|_{H^{s-\f12}}^2d\tau+C(L_1,L_2)\int_0^t\mathcal{E}_s(\tau)d\tau\Big).
\end{align*}
The desired estimate (\ref{mhd51}) follows from Lemma \ref{lem:non-g} and Gronwall's inequality.
\end{proof}

For the vorticity and current system (\ref{eq:vorticity-w})-(\ref{eq:vorticity-h}),
we introduce the following linearized system:
\begin{align}\label{eq:vorticity-w-L}
&\p_t\bar\vom+\vu\cdot\nabla\bar\vom-\vh\cdot\nabla\bar\vj
=\bar\vom\cdot\nabla\vu-\bar\vj\cdot\nabla\vh \text{ in }Q_T,\\
\label{eq:vorticity-h-L}
&\p_t\bar\vj+\vu\cdot\nabla\bar\vj-\vh\cdot\nabla\bar\vom
=\bar\vj\cdot\nabla\vu-\bar\vom\cdot\nabla\vh-2\nabla u_i\times\nabla h_i \text{ in }Q_T,
\end{align}
together with the initial data
\beq\label{mhd53}
\bar\vom(0,x) = \bar\vom_0(x), \quad \bar\vj(0,x) = \bar\vj_0(x), \, x\in \Om_f.
\eeq
The following proposition can be proved in a standard way(see \cite{SWZ}).
\begin{proposition}\label{prop:vorticity}
Given $\bar\om_0,\bar\vj_0\in H^{s-1}(\Om_f)$, there exists a unique solution $(\bar\vom,\bar\vj)$ to the initial value problem (\ref{eq:vorticity-w-L})-(\ref{mhd53}) such that
\begin{align*}
\sup_{t\in[0,T]}&\Big(\|\bar\vom(t)\|^2_{H^{s-1}(\Omega_f)}+\|\bar\vj(t)\|^2_{H^{s-1}(\Omega_f)}\Big) \le \Big(1+\|\bar\vom_0\|_{H^{s-1}(\Omega_0)}^2+\|\bar\vj_0\|_{H^{s-1}(\Omega_0)}^2\Big)e^{C(L_1)T}.
\end{align*}
Moreover, it holds that
\begin{align*}
\frac{d}{dt}\int_{\Gamma}\bar\om_3dx'=0,\quad \frac{d}{dt}\int_{\Gamma}\bar\xi_3dx'=0.
\end{align*}
\end{proposition}

Finally, the magnetic field $\hat\vh$ in the vacuum is considered as a secondary variable computed from $\Gamma_{f}$ and $\hat{\mathbf{J}}$ by solving the following div-curl system:
\begin{equation}\label{mhdhath}
\left\{\begin{aligned}
&\curl {\hat{\vh}}=0,\,\,\div {\hat{\vh}} = 0 \text{ in } \Om_{f}^+,\\
&{\hat{\vh}}\cdot\vN_{f} = 0  \text{ on }\,\Gamma_{f}, \,  {\hat{\vh}}\times\ve_3 =\hat{\mathbf{J}} \text{ on } \Gamma^+,
\end{aligned}\right.
\end{equation}
for any fixed time $t\ge 0$. According to Proposition \ref{prop:div-curl-2},
\beq\label{mhd55}
\left\|{\hat{\vh}}(t)\right\|_{H^{s}(\Omega_{f}^+)} \le C\left(\|f(t)\|_{H^{s+\f12}}\right)\|\hat{\mathbf{J}}(t)\|_{H^{s-\f12}}.
\eeq

\section{Construction of the iteration map}\label{s6}
Given $f_0\in H^{s+\f12}(\BT),\,  \vu_0,\,\vh_0\in H^{s}(\Omega_{f_0})$ and $\hat{\mathbf{J}}\in C^k\big([0,T_0];H^{s-\f12-k}\big), k=0,1$, we first solve $\hat{\vh}_0\in H^s(\Om_{f_0}^+)$ from (\ref{mhd9}) as mentioned before, and assume furthermore that there exists $c_0,c_1>0$ so that
\beq\label{mhd61}
 -1 + \f{c_0}{2}\le f_0(x')\le 1 - \f{c_0}{2},  \quad \Lambda(\vh_0,\hat\vh_0)\ge 2c_1.
\eeq

We then choose $f_*=f_0$ and take $\Omega_*=\Omega_{f_0}$ as the reference region.
The initial data $(f_I,(\p_tf)_I,\vom_{*I},\vj_{*I},\beta_{Ii},
\gamma_{Ii})$ for the equivalent system formulated at the beginning of Section \ref{s3} is defined as follows:
\begin{align*}
&f_I=f_0,\,\, (\p_tf)_I=\vu_0(x',f_0(x'))\cdot(-\partial_1f_0,-\partial_2f_0,1),\\
&\vom_{*I}=\curl\vu_0,\,\, \vj_{*I}=\curl\vh_0,\\
&\beta_{Ii}=\int_{\BT}u_{0i}(x',-1)dx',\,\, \gamma_{Ii}=\int_{\BT}h_{0i}(x',-1)dx'.
\end{align*}
In addition, we choose a large constant $M_0>1$ so that
\begin{align}
&\|f_I\|_{H^{s+\f12}}+\|(\vom_{I*}, \vj_{I*})\|_{H^{s-1}(\Omega_*)}
+\|(\p_tf)_I\|_{H^{s-\f12}} + |\beta_{Ii}|+|\gamma_{Ii}|+ \|\hat{\vh}_0\|_{H^{s}(\Om_*)} \\ \nonumber
\!\!& + \|\hat{\mathbf{J}}(t)\|_{H^{s-\f12}(\Gamma^+)} + \|\p_t\hat{\mathbf{J}}(t)\|_{H^{s-\f32}(\Gamma^+)} \le M_0.
\end{align}

The iteration map we constructed is essentially based on iterating the unknowns $f, \vom_*, \vj_*, \beta_{i},\gamma_{i}$, which satisfy certain evolution equations, while the secondary variable $\hat\vh$ is determined by $(f,\hat{\mathbf{J}})$ through (\ref{mhdhath}).

Now we introduce the following functional space.
\begin{definition}\label{def:X}
Given two positive constants $M_1, M_2>0$ with $M_1\ge 2M_0$, we define
the space $\mathcal{X}=\mathcal{X}(T, M_1, M_2)$ as the collection of
$(f, \vom_*, \vj_*, \beta_{i},\gamma_{i})$, which satisfies
\begin{eqnarray}\label{defM1}
&\left(f(0),\p_tf(0), \vom_*(0), \vj_*(0),\beta_{i}(0),\gamma_{i}(0)\right)=\big(f_I, (\p_t f)_I, \vom_{*I}, \vj_{*I}, \beta_{iI},\gamma_{iI}\big),&\\
\label{defM2}
&\sup\limits_{t\in[0,T]}
\|f(t,\cdot)-f_*\|_{H^{s-\f12}} \le \delta_0,&\\
~~~~&~~~~\sup\limits_{t\in[0,T]}\Big(\|f(t)\|_{H^{s+\f12}}+\|\p_t f(t)\|_{H^{s-\f12}}
+\|(\vom_*, \vj_*)(t)\|_{H^{s-1}(\Omega_*)}
 + |\beta_{i}(t)|+|\gamma_{i}(t)| \Big)\le M_1,&\\
&\sup\limits_{t\in[0,T]}\Big(\|\partial_{t}^2f(t)\|_{H^{s-\f32}}
+\|(\p_t\vom_*, \p_t\vj_*)(t)\|_{H^{s-2}(\Omega_*)}
+ |\p_t\beta_{i}(t)|+|\p_t\gamma_{i}(t)| \Big)\le M_2,&
\end{eqnarray}
together with the following compatibility conditions:
\beq\label{Mcomp}
\int_{\T^2}\pa_t f(t,x')dx'=0,\,\, \int_{\Gamma} \vom_{*3} dx' = \int_{\Gamma} \vj_{*3} dx' = 0.
\eeq
\end{definition}

Given $(f,\vom_*, \vj_*, \beta_{i},\gamma_{i} )\in \mathcal{X}
(T, M_1, M_2)$, our goal is to construct an iteration map
$$(\bar{f},\bar{\vom}_*, \bar{\vj}_*, \bar{\beta}_{i},\bar{\gamma}_{i}) = \mathcal{F}\big(f,\vom_*, \vj_*, \beta_{i},\gamma_{i} \big)
\in\mathcal{X}(T, M_1, M_2)$$
with suitably chosen constants $M_1, M_2$ and $T$.

\subsection{Recover the bulk region, velocity and magnetic fields}\label{s61}
Recall that
\begin{align*}
\Om_f^{+}=\left\{ x \in \Om| x_3 > f (t,x')\right\}, \,\Om_f \left(= \Om_f^{-}\right) = \left\{ x \in \Om| x_3 < f (t,x')\right\},
\end{align*}
and the harmonic coordinate map $\Phi_f:\Omega_*\to\Omega_f$ defined in (\ref{sys:har-coordinate}).

We first define $\hat\vh$ by solving (\ref{mhdhath}). Then $\p_t\hat\vh$ satisfies
\begin{equation}\label{mhd63}
\left\{\begin{aligned}
&\curl\p_t {\hat\vh}=0,\,\,\div\p_t {\hat\vh}=0 \text{ in } \Om_{f}^+,\\
&\p_t {\hat\vh}\cdot\vN_{f}=-\p_t f\partial_3 {\hat\vh}
\cdot\vN_{f} + {\hat{h}}_1\partial_1\p_t {f} + {\hat{h}}_2\partial_2\p_t {f}
 \text{ on }\Gamma_{f}, \\
&\p_t {\hat\vh}\times\ve_3=\p_t\hat{\mathbf{J}} \text{ on } \Gamma^+.
\end{aligned}\right.
\end{equation}
It follows from Proposition \ref{prop:div-curl-2} that
\beq\label{mhd65}
\|{\hat\vh}(t)\|_{H^{s}(\Om_{f}^+)} + \|\p_t {\hat\vh}(t)\|_{H^{s-1}(\Om_{f}^+)} \le C(M_0,M_1).
\eeq

To recover $\vu,\vh$ in $\Om_f$, we define an operator which projects any vector field in
$\Omega_{f}$ to its divergence-free part. More precisely, for any $\vom\in H^s(\Om_f)$, let
$P_{f }^{\div}\vom=\vom-\nabla\phi$ with $\phi$ solving the following mixed boundary value problem:
\beq
\Delta\phi=\div\vom\,\,\text{in}\, \Omega_f,\,\,
\partial_3\phi=0 \text{ on }  \Gamma,\,\, \phi=0 \text{ on }\, \Gamma_f.
\eeq
We denote
\begin{align*}
\widetilde{\vom} = P_{f}^{\div} (\vom _*\circ\Phi_{f }^{-1}),\quad
\widetilde{\vj} = P_{f}^{\div} (\vj _*\circ\Phi_{f }^{-1}).
\end{align*}
It follows from Lemma \ref{lem:basic} for harmonic coordinates and standard elliptic estimates that
\begin{align}\label{eq:wh-est1}
\|(\widetilde{\vom}, \widetilde{\vj})\|_{H^{s-1}(\Omega_f)}\le C(M_1),\quad \|(\pa_t\widetilde{\vom}, \pa_t\widetilde{\vj})\|_{H^{s-2}(\Omega_f)}\le C\big(M_1, M_2\big).
\end{align}
Moreover, since $\div \widetilde{\vom}=0$ and $\widetilde{\vom}_3 = \ve_3\cdot\widetilde{\vom} = \ve_3\cdot \vom_*=\om_{*3}$ on $\Gamma$,
$\widetilde\vom$ satisfies compatibility conditions in (\ref{mhd41}) according to (\ref{Mcomp}). Similar argument applies to $\widetilde{\vj}$.
Then we can define the velocity field $\vu$ and magnetic field $\vh$ in $\Omega_f$ by solving the following div-curl system
\begin{equation}\label{mhd68}
\left\{\begin{aligned}
&\curl\vu=\widetilde{\vom},\,\,\div\vu =0  \text{ in } \Om_f,\\
&\vu\cdot\vN_f=\p_tf \text{ on } \Gamma_{f}, \,\, \vu\cdot\ve_3 = 0,\,\,\int_{\Gamma}u_i dx'=\beta_i \text{ on } \Gamma,\, i=1,2,
\end{aligned}\right.
\end{equation}
and
\begin{equation}
\left\{\begin{aligned}
&\curl \vh=\widetilde{\vj},\,\,\div\vh=0 \text{ in } \Om_f,\\
&\vh\cdot\vN_f = 0 \text{ on } \Gamma_{f}, \,\, \vh\cdot\ve_3 = 0, \,\, \int_{\Gamma} h_i dx'=\gamma_i \text{ on } \Gamma, \, i=1,2.
\end{aligned}\right.
\end{equation}
It follows from Proposition \ref{prop:div-curl} and (\ref{eq:wh-est1}) that
\begin{align}
\|\vu\|_{H^{s}(\Omega_f)}\le &C(M_1)\big(\|\widetilde{\vom}\|_{H^{s-1}(\Omega_f)}+\|\p_tf
\|_{H^{s-\f12}}+|\beta_1|+|\beta_2|\big)\le C(M_1),\\
\|\vh\|_{H^{s}(\Omega_f)}\le &C(M_1)\big(\|\widetilde{\vj}\|_{H^{s-1}(\Omega_f)}+|\gamma_1|+|\gamma_2|\big)\le C(M_1).
\end{align}
In addition,
\begin{align*}
\vu(0)=\vu_0,\,\,\, \vh(0)=\vh_0.
\end{align*}
Using the fact that on $\Gamma_f$,
\begin{align*}
\p_t(\vu\cdot\vN_f)=(\p_t\vu+\partial_3\vu\p_tf)\cdot\vN_f+\vu\cdot\p_t\vN_f,
\end{align*}
we deduce that
\begin{equation}\nonumber
\left\{\begin{aligned}
&\curl\p_t\vu=\p_t\widetilde{\vom},\,\,\,\div\p_t\vu=0 \text{ in } \Om_f,\\
&\p_t\vu\cdot\vN_f=\partial_{t}^2f-\p_tf\partial_3\vu
\cdot\vN_f+u_1\partial_1\p_tf+u_2\partial_2\p_tf
\text{ on } \Gamma_{f}, \\
&\p_t\vu\cdot\ve_3=0, \, \int_{\Gamma}\p_tu_i dx=\p_t\beta_i\text{ on } \Gamma, \,i=1,2.
\end{aligned}\right.
\end{equation}
By Proposition \ref{prop:div-curl} and (\ref{eq:wh-est1}) again,
\begin{align*}
\|\p_t\vu\|_{H^{s-1}(\Omega_f)}\le {C}(M_1, M_2),
\end{align*}
which implies
\begin{align}\nonumber
\|\vu(t)\|_{L^\infty(\Gamma_f)}\le& \|\vu_0\|_{L^\infty(\Gamma_{f_0})}+\int_0^t\|\p_t\vu\|_{L^\infty(\Gamma_f)} dt
\le {M_0}+T{C}(M_1,M_2).
\end{align}
Similarly,
\begin{align*}
\|\p_t\vh(t)\|_{H^{s-1}(\Omega_f)}\le {C}(M_1, M_2),\quad
\|\vh(t)\|_{L^\infty(\Gamma_f)}\le  {M_0} + T{C}(M_1,M_2).
\end{align*}
Also, we have
\beno
\|f(t)-f_0\|_{L^\infty}\le \|f(t)-f_0\|_{H^{s-\f12}}\le T\|\p_tf\|_{H^{s-\f12}}\le T M_1,
\eeno
as well as
\beno
|\Lambda(\vh,\hat\vh)-\Lambda(\vh_0,\hat\vh_0)|\le TC\big(\|\p_t\vh\|_{L^\infty(\Gamma_f)},
\|\p_t\hat{\vh}\|_{L^\infty(\Gamma_f)}\big)\le T C(M_0,M_1,M_2).
\eeno
Without loss of generality, we assume $M_1\le  C(M_1)\le C(M_0, M_1)\leq C(M_1,M_2)\leq C(M_0,M_1,M_2)$ and choose $T\le \min\{1,T_0\}$ small enough so that
\begin{align*}
T C(M_0,M_1,M_2)\le \min\{\delta_0,c_1\} (\le {M_0}).
\end{align*}
Let $L_0=2M_0$, $L_1= 10 C(M_0, M_1)$, $L_2= 10 {C}(M_0,M_1, M_2)$. We conclude that for any $t\in [0,T]$,
\begin{itemize}\label{Lbound}
\item $\|(\vu, \vh,\hat{\vh})(t)\|_{L^{\infty}(\Gamma_f)}\le L_0$;
\item $\|f(t)\|_{H^{s+\f12}}+\|\pa_t f(t)\|_{H^{s-\f12}}+\|(\vu,\vh)(t)\|_{H^{s}(\Omega_f)}
+\|\hat{\vh}(t)\|_{H^{s}(\Omega_f^+)} + \|\p_t\hat{\vh}(t)\|_{H^{s-1}(\Om_f^+)}\le L_1$;
\item $\|(\p_t\vu, \p_t\vh)(t)\|_{L^{\infty}(\Gamma_f)} \le L_2$;
\item $-(1-c_0)\le f(t,x')\le (1-c_0), \,\, \|f(t)-f_*\|_{H^{s-\f12}}\le \delta_0$;
\item $\Lambda(\vh,\hat\vh)(t)\ge c_1$;
\end{itemize}
which are nothing but the bounds listed at the beginning of Section \ref{s5}.

\subsection{Define the iteration map}

Given $(f,\vu,\vh,\hat{\vh})$ as above, we define the iteration map.
We first solve $\bar f_1$ by the linearized system (\ref{sys:linear-H}) and then $(\bar\vom, \bar\vj)$ by (\ref{eq:vorticity-w-L}) and
(\ref{eq:vorticity-h-L}) with the initial data
\begin{align*}
&\left(\bar f_1(0),\bar\theta(0), \bar\vom(0) , \bar\vj(0)\right)
=\big(f_0,(\p_tf)_I, \vom_{*I}, \vj_{*I}\big).
\end{align*}
We define
\begin{align*}
&\bar\vom_{*}=\bar\vom\circ \Phi_{f},\, \bar\vj_{*}=\bar\vj\circ\Phi_{f},\\
&\bar\beta_i(t)=\beta_i(0)-\int_0^t\int_{\Gamma}\big(u_j\partial_ju_i-h_j\partial_jh_i\big)dx'd\tau,\\
&\bar\gamma_i(t)=\gamma_i(0)-\int_0^t\int_{\Gamma}\big(u_j\partial_jh_i-h_j\partial_ju_i\big)dx'd\tau,\\
&\bar f(t,x')=\bar f_1(t,x')-\langle \bar f_1\rangle+\langle f_0\rangle.
\end{align*}
Note that $\langle \bar f\rangle=\langle f_0\rangle$ and $\int_{\T^2}\pa_t \bar f(t,x')dx'=0$ for $t\in [0,T]$. Moreover, according to Proposition \ref{prop:vorticity},
\[
\int_{\Gamma} \bar{\vom}_{*3} dx' = \int_\Gamma \bar{\vom}_3 dx' = 0, \quad \int_{\Gamma} \bar{\vj}_{*3} dx' = \int_{\Gamma} \bar{\vj}_{3} dx' = 0.
\]
The iteration map $\mathcal{F}$ is defined as follows
\ben
\mathcal{F}\big(f,\vom_*, \vj_*, \beta_{i},\gamma_{i} \big)
\eqdefa \big(\bar{f},\bar{\vom}_*, \bar{\vj}_*,\bar{\beta}_{i},
\bar{\gamma}_{i} \big).
\een

\begin{proposition}\label{prop:iteration map}
There exist $M_1, M_2, T>0$ depending on $c_0, c_1, \delta_0, M_0$ so that $\mathcal{F}$ is a map from $\mathcal{X}(T, M_1,M_2)$ to itself.
\end{proposition}
\begin{proof}
First note that the initial conditions is automatically satisfied. Proposition \ref{prop:f-L} and Proposition \ref{prop:vorticity}
ensure that for any $t\in[0,T]$,
\beno
\left( \|\bar{f}(t)\|_{H^{s+\f12}}+\|\p_t\bar{f}(t)\|_{H^{s-\f12}}+\|\bar\vom_*(t)\|_{H^{s-1}
(\Omega_*)}+\|\bar\vj_*(t)\|_{H^{s-1}(\Omega_*)}\right)\le C(c_0,M_0)e^{C(M_1,M_2)T}.
\eeno
From the equation (\ref{sys:linear-H}), (\ref{eq:vorticity-w-L}) and (\ref{eq:vorticity-h-L}) together with (\ref{mhd65}) for $\hat\vh$, we deduce that
\beno
\sup_{t\in[0,T]}\Big(\|\partial_{t}^2\bar f\|_{H^{s-\f32}}
+\|(\p_t\vom_*, \p_t\vj_*)\|_{H^{s-2}(\Omega_*)}\Big)\le C(M_1).
\eeno
Moreover, it is obvious that
\beno
&&|\p_t\bar\beta_i(t)|+|\p_t\bar\gamma_i(t)|\le C(M_1),\\
&&|\bar\beta_i(t)|+|\bar\gamma_i(t)|\le M_0+T C(M_1),\\
&&\|\bar f(t)-f_*\|_{H^{s-\f12}}\le \int_0^t\|\pa_t \bar f(\tau)\|_{H^{s-\f12}}d\tau \le T C(M_1).
\eeno

We take $M_1=2\max\{M_0,C(c_0,M_0)\}$ and $M_2=C(M_1)$.
Finally, let $T$ be sufficiently small depending only on $c_0, c_1, \delta_0, M_0$ so that all other conditions in Definition \ref{def:X} are satisfied.
\end{proof}

\section{Contraction of the Iteration Map}\label{s7}

\subsection{Contraction}

Let $\big(f^A, \vom_*^{ A}, \vj_*^{ A},\beta^{ A}_{i}, \gamma^{ A}_{i}\big)$ and
$\big(f^B$, $\vom_*^{ B}, \vj_*^{ B}, \beta^{ B}_{i}, \gamma^{ B}_{i}\big)$
be two elements in $\mathcal{X}(T, M_1,M_2)$, and
$\big(\bar f^C,\bar \vom_*^{ C}, \bar\vj_*^{ C}, \bar\beta^{ C}_{i},\bar\gamma^{ C}_{i}\big)=\mathcal{F}
\big(f^C$, $\vom_*^{ C}, \vj_*^{ C},\beta^{ C}_{i}, \gamma^{ C}_{i}\big)$ for $C=A,B$. Correspondingly, we have quantities $\vu^C,\vh^C$ and $\hat{\vh}^C$ defined in $\Om_{f^C}$ and $\Om_{f^C}^+$ respectively.
For a quantity $q$, we denote by $q^D$ the difference $q^A-q^B$.

\begin{proposition}\label{prop:contraction}
There exists $T>0$ depending on $c_0, \delta_0, M_0$ so that
\begin{align}\nonumber
\bar E^D_s &:= \sup_{t\in[0,T]}\Big(\|\bar{f}^D(t)\|_{H^{s-\f12}}+\|\p_t\bar{f}^D(t)\|_{H^{s-\f32}}+\|\bar{\vom}_*^{ D}(t)\|_{H^{s-2}(\Omega_*)}
\\&\qquad\qquad+\|\bar{\vj}_*^{ D}(t)\|_{H^{s-2}(\Omega_*)}+|\bar\beta^{ D}_{i}(t)|+|\bar\gamma^{ D}_{i}(t)|\Big)\nonumber\\\nonumber
&\le\frac12\sup_{t\in[0,T]}\Big( \|{f}^D(t)\|_{H^{s-\f12}}+\|\p_t{f}^D(t)\|_{H^{s-\f32}}
+\|\vom_*^{ D}(t)\|_{H^{s-2}(\Omega_*)}\\&\qquad\qquad+\|\vj_*^{ D}(t)\|_{H^{s-2}(\Omega_*)}
+|\beta^{ D}_{i}(t)|+|\gamma^{ D}_{i}(t)|\Big):= E^D_s.\nonumber
\end{align}
\end{proposition}

\begin{proof}
First of all, by the elliptic estimates, we have
\begin{align*}
\|\Phi_{f^A}-\Phi_{f^B}\|_{H^{s}(\Om_*)}\le C(M_1)\|f^A-f^B\|_{H^{s-\f12}}\le C E^D_s.
\end{align*}
Due to the fact that $\vu^A$ and $\vu^B$ are defined on different regions, one can not estimate their difference directly.
To this end, we introduce for $C=A,B$,
$$
\vu^{ C}_*=\vu^{ C}\circ\Phi_{f^C},\,\,
\vh^{ C}_*=\vh^{ C}\circ\Phi_{f^C}, \,\, \hat{\vh}^C_* = \hat{\vh}^C\circ\Phi_{f^C}^+.
$$

We first show that
\begin{align}\label{eq:uh-d}
\left\|\left(\vu^{ D}_*,\vh^{ D}_*\right)\right\|_{H^{s-1}(\Om_*)} + \left\|\hat{\vh}^{ D}_* \right\|_{H^{s-1}(\Om_*^+)}\le CE^D_s.
\end{align}
For a vector field $\vv_*$ defined on $\Omega_*$, we define
\begin{align*}
\curl_C \vv_*=\big(\curl (\vv_*\circ(\Phi_{f^C})^{-1})\big) \circ\Phi_{f^C},\,\,
\div_C \vv_* =\big(\div(\vv_*\circ(\Phi_{f^C})^{-1}\big) \circ\Phi_{f^C},
\end{align*}
for $C=A,B$. Then we find by (\ref{mhd68}) that for $C=A,B$,
\begin{equation}\nonumber
\left\{\begin{aligned}
&\curl_C \vu^{ C}_*=\widetilde{\vom}^{ C}_*,\,\,\, \div_C \vu_*^{ C}=0\,\,\text{ in } \Om_*,\\
&\vu^{ C}_*\cdot\vN_{f^C}=\p_tf^C \text{ on } \Gamma_{*},\,\, \vu^{ C}\cdot\ve_3 = 0,\, \int_{\Gamma} u_i^{ C}dx'=\beta_i^{ C} \text{ on } \Gamma.
\end{aligned}\right.
\end{equation}
Thus, we obtain
\begin{equation}\nonumber
\left\{\begin{aligned}
&\curl_A\vu^{ D}_*=\widetilde{\vom}^{ D}_*+(\curl_B-\curl_A)\vu^{ B}_*\,\,\text{ in }  \Om_*,\\
&\div_A\vu^{ D}_*=(\div_B-\div_A)\vu^{ B}_*\,\,\text{ in }\Om_*,\\
&\vu^{ D}_*\cdot\vN_{f^A}=\p_tf^D+\vu^{ B}_*\cdot(\vN_{f^B}-\vN_{f^A})\,\, \text{ on } \Gamma_{*}, \\
&\vu^{ D}_*\cdot\ve_3=0,\, \int_{\Gamma}u_i^{ D}dx'=\beta_i^{ D} \text{ on } \Gamma.
\end{aligned}\right.
\end{equation}
A tedious but direct calculation shows that
\begin{align*}
\|(\curl_B-\curl_A)\vu^{ B}_*\|_{H^{s-2}(\Omega_*)}\le&~ C\|\Phi_{f^A}-\Phi_{f^B}\|_{H^{s-1}(\Omega_*)}
\le C\|f^D\|_{H^{s-\f12}}\le CE^D_s.
\end{align*}
Similarly,
\begin{align*}
\|(\div_B-\div_A)\vu^{ B}_*\|_{H^{s-2}(\Omega_*)}\le CE^D_s, \quad \|\vu^{ B}_*\cdot(\vN_{f^B}-\vN_{f^B})\|_{H^{s-\f32}}\le CE^D_s.
\end{align*}
We deduce from Proposition \ref{prop:div-curl} that
\begin{align}\nonumber
\|\vu^{ D}_*\|_{H^{s-1}(\Om_*)}\le C\left(\|\widetilde{\vom}^{ D}_*\|_{H^{s-2}(\Omega_*)}
+\|\p_tf^D\|_{H^{s-\f32}}+E^D\right)\le CE^D_s.
\end{align}
By applying similar arguments to $\vh,\hat\vh$, we have from Proposition \ref{prop:div-curl}-\ref{prop:div-curl-2} that
\begin{align}\nonumber
\|\vh^{ D}_*\|_{H^{s-1}(\Om_*)}\le CE^D_s, \quad \|\hat{\vh}^{ D}_*\|_{H^{s-1}(\Om_*^+)}\le CE^D_s.
\end{align}
Thus, we conclude the proof of (\ref{eq:uh-d}).

To estimate $f^D$, we first note that
\beq\label{linear:cauchy:eq-2}\left\{\begin{aligned}
&\p_t\bar{f}_1^D = \bar\theta^D,\\
&\p_t\bar{\theta}^D = 2\left(\ul{u^A_1}\partial_1\bar\theta^D+\ul{u^A_2}\partial_2\bar\theta^D\right)+\sum_{i,j=1,2}
\left(-\ul{u^A_i}\ul{u^A_j}+\ul{h^A_i}\ul{h^A_j}+\ul{\hat{h}^A_i}\ul{\hat{h}^A_j}\right)\partial_i\partial_j\bar{f}^D
+\mathfrak{R},
\end{aligned}\right.
\eeq
where
\begin{align*}
\mathfrak{R}= &-2\left(\left(\ul{u^{D}_1} + \ul{u^{D}_1}\right)\partial_1\bar{\theta}^B
+\left(\ul{u^{D}_2}+\ul{u^{D}_2}\right)\partial_2\bar\theta^B\right)\\
&+\sum_{i,j=1,2}\left(\left(-\ul{u^A_i}\ul{u^A_j}+\ul{h^A_i}\ul{h^A_j}+\ul{\hat{h}^A_i}\ul{\hat{h}^A_j}
\right)-\left(-\ul{u^B_i}\ul{u^B_j}+\ul{h^B_i}\ul{h^B_j}+\ul{\hat{h}^B_i}\ul{\hat{h}^B_j}
\right)\right)\partial_i\partial_j\bar{f}_1^B\\
&+\mathfrak{g}^A-\mathfrak{g}^B.
\end{align*}
Here for $C=A,B$,
\begin{align*}
\mathfrak{g}^C=
&-\frac12\vN_{f^C}\cdot\ul{\nabla(p_{\vu^C,\vu^C}-p_{\vh^C,\vh^C})}+\frac12\vN_{f^C}\cdot\ul{\nabla(|\hha^C|^2-\hat{\mathcal{H}}^+_{f^C} |\hha^C|^2)}+\frac12(\hat{\mathcal{N}}^+_{f^C}-\mathcal{N}_{f^C})|\hat{\vh}^C|^2.
\end{align*}
It is direct to verify that
\beno
\|\mathfrak{R}\|_{H^{s-\f32}}\le CE^D_s.
\eeno
We denote
\begin{align*}
\bar F_{s}^D(\p_t\bar{f}_1^D,\bar{f}_1^D)
=& \big\|(\p_t+u^A_i\partial_i)\langle \na\rangle^{s-\f32}\bar{f}_1^D\big\|_{L^2}^2
+\frac12\big\|h^{A}_i\partial_i\langle \na\rangle^{s-\f32}\bar{f}_1^D\big\|_{L^2}^2\\
&+\frac12\big\|\hat{h}^{A}_i\partial_i\langle \na\rangle^{s-\f32}\bar{f}_1^D\big\|_{L^2}^2.
\end{align*}
A similar argument as in Proposition \ref{prop:f-L} gives
\begin{align*}
\frac{d}{dt}\Big(\bar F_{s}^D(\p_t\bar{f}_1^D,\bar{f}_1^D)+\|\bar{f}_1^D\|_{L^2}^2+\|\p_t\bar{f}_1^D\|_{L^2}^2\Big)
\le  C\big(E^D_s+\bar E_{1s}^D).
\end{align*}
where
\begin{align*}
\bar E_{1s}^D&~=\sup_{t\in[0,T]}\Big(\|\bar{f}_1^D(t)\|_{H^{s-\f12}}+\|\p_t\bar{f}_1^D(t)\|_{H^{s-\f32}}\Big).
\end{align*}
Stability condition on $\vh^A,\hat\vh^A$ implies that
\begin{align*}
\|\bar{f}_1^D\|_{H^{s-\f12}}^2+\|\p_t\bar{f}_1^D\|_{H^{s-\f32}}^2\le C
\Big(\bar F_{s}^D(\bar{f}_1^D,\p_t\bar{f}_1^D)+\|\bar{f}_1^D\|_{L^2}^2
+\|\p_t\bar{f}_1^D\|_{L^2}^2\Big).
\end{align*}
It follows that
\begin{align*}
\bar{E}_{1s}\le CTE^D_s,
\end{align*}
which implies
\begin{align}\label{eq:f-d}
\sup_{t\in[0,T]}\left(\|\bar{f}^D(t)\|_{H^{s-1}}+\|\p_t\bar{f}^D
(t)\|_{H^{s-\f32}}\right)\le CTE^D_s.
\end{align}

Similar to the proof of Proposition \ref{prop:vorticity}, one can show that
\begin{align}\label{eq:wj-d}
\sup_{t\in[0,T]}\left(\|\bar\vom_*^D(t)\|_{H^{s-2}(\Omega_*)}+\|\bar\vj_*^D\|_{H^{s-2}(\Omega_*)}\right)\le CTE^D_s.
\end{align}

Finally, by using the equation
\begin{align*}
\bar\beta^{ C}_i(t)=\beta^{ C}_i(0)+\int_0^t\int_{\Gamma}u_j^{ C}\big(\partial_ju_i^{ C}-h^{ C}_j\partial_jh^{ C}_i\big)dx'd\tau,
\end{align*}
it is obvious that
\begin{align}\label{eq:be-d}
|\bar\beta^{ D}_i(t)|\le|\beta^{ D}_{iI}|+CTE^D_s = CTE^D_s.
\end{align}
Similarly,
\begin{align}\label{eq:ga-d}
|\bar\gamma^{ D}_i(t)|\le|\gamma^{ D}_{iI}|+CTE^D_s = CTE^D_s.
\end{align}

We deduce from (\ref{eq:uh-d}) and (\ref{eq:f-d})--(\ref{eq:ga-d}) that
\begin{align*}
\bar{E}^D_s &\le CTE^D_s.
\end{align*}
The proof is concluded by taking $T=\f{1}{2C}$ with $C$ depending only on $c_0, \delta_0, M_0$.
\end{proof}

\subsection{The limit system}

Proposition \ref{prop:iteration map} and Proposition \ref{prop:contraction} ensure that the map $\mathcal{F}$ has a unique fixed point
$(f,\vom, \vj,\beta_i, \gamma_i)$
in $\mathcal{X}(T,M_1,M_2)$. From the construction of $\mathcal{F}$, we know that $(f,\vom, \vj,\beta_i, \gamma_i)$
satisfies
\beq\label{eq:limit-theta}\left\{\begin{aligned}
\p_tf&~=\theta-\langle\theta\rangle,\\
\p_t\theta&~=
2(\ul{u_1}\partial_1\theta+\ul{u_2}\partial_2\theta\big)-\sum_{i,j=1,2}
\big(\ul{u_i}\ul{u_j}-\ul{h_i}\ul{h_j}
-\ul{\hat h_i}\ul{\hat h_j}\big)\partial_i\partial_jf\\
& \quad-\vN_f \cdot \ul{\nabla(p_{\vu,\vu}-p_{\vh,\vh})}+\frac12\vN_f\cdot\ul{\nabla(|\hha|^2-\hat{\mathcal{H}}_f^+ |\hha|^2)}
+\frac12(\hat{\mathcal{N}}_f^+-\mathcal{N}_f)|\hat{\vh}|^2,
\end{aligned}\right.\eeq
where $(\vu,\vh)$ solves the div-curl system
\beq\label{mhd72}
\left\{
\begin{array}{l}
\curl \vu=P_f^{div}\vom, \, \div\vu=0  \text{ in }  \Om_f,\\
\vu\cdot{\vN_f} = \p_t f \text{ on }\Gamma_f,\, u_3=0\text{ on }\Gamma,\\
\int_{\Gamma}u_i dx'=\beta_i,\,
\p_t\beta_i~=-\int_{\Gamma}(u_j\partial_ju_i-h_j\partial_jh_i)dx',\, i=1,2,
\end{array}\right.\eeq
and
\beq\label{mhd73}\left\{
\begin{array}{l}
\curl\vh=P_f^{div}\vj,\, \div \vh=0\, \text{in} \, \Om_f,\\
\vh\cdot\vN_f=0\, \text{ on } \Gamma_f,\,h_3=0\text{ on }\Gamma,\\
\int_{\Gamma}h_i dx'=\gamma_i,\,
\p_t\gamma_i=-\int_{\Gamma}(u_j\partial_jh_i-h_j\partial_ju_i)dx',\, i=1,2,
\end{array}\right.
\eeq
and $\hat\vh$ solves the div-curl system
\beq\label{mhd74}
\left\{
\begin{array}{l}
\curl\hat{\vh}=0,\, \div \hat{\vh}=0 \text{ in }  \Om_f^+,\\
\hat\vh\cdot\vN_f = 0 \text{ on } \Gamma_f,\,
\hat\vh\times\mathbf{e}_3 = \hat{\mathbf{J}}\text{ on }\Gamma^+,
\end{array}\right.
\eeq
as well as
\beq\label{eq:limit-h}\left\{\begin{aligned}
&\p_t\vom+\vu\cdot\nabla\vom=\vh\cdot\nabla\vj+\vom\cdot\nabla\vu-\vj\cdot\nabla\vh\,\, \text{ in }Q_T,\\
&\p_t\vj+\vu\cdot\nabla\vj=\vh\cdot\nabla\vom+\vj\cdot\nabla\vu
-\vom\cdot\nabla\vh-2\nabla u_i\times\nabla h_i\,\, \text{ in }Q_T.
\end{aligned}\right.\eeq
We also recall that $p_{\vv, \vv}$ for $\vv=\vu,\vh$ is determined by the elliptic equation
\beq\label{eq:limit-pressure}
\left\{
\begin{aligned}
&\Delta p_{\vv, \vv}= -\mathrm{tr}(\nabla\vv\nabla\vv)\,\, \text{ in }\, \Omega_f,\\
&p_{\vv, \vv}=0\text{ on }\Gamma_f, \,\ve_3\cdot\nabla p_{\vv, \vv}=0\text{ on }\,\Gamma.
\end{aligned}\right.
\eeq

\section{From the limit system to the plasma-vacuum interface system}\label{s8}

It is not obvious whether the limit system (\ref{eq:limit-theta})-(\ref{eq:limit-pressure}) is equivalent to the plasma-vacuum interface system (\ref{mhd01})-(\ref{mhd7}). Following \cite{SWZ}, we split the proof into several steps.

{\bf Step 1.} $\curl\vu=\vom$ and $\curl\vh=\vj$.

By $\div\vu=\div\vh=0$, it is easy to verify that
\begin{align*}
&\p_t\div\vom+\vu\cdot\nabla\div\vom=\vh\cdot\nabla\div\vj\,\, \text{ in }Q_T,\\
&\p_t\div\vj+\vu\cdot\nabla\div\vj=\vh\cdot\nabla\div\vom\,\, \text{ in }Q_T,
\end{align*}
which imply that $\div\vom=\div\vj=0$ since it is satisfied initially. Hence $\curl\vu=\vom$, $\curl\vh=\vj$ according to (\ref{mhd72}) and (\ref{mhd73}).
\smallskip

{\bf Step 2.} Determination of the pressure.

Let the pressure $p$ in the plasma region be given by
\begin{align}
p=\frac12\mathcal{H}|\ul{\hat\vh}|^2+p_{\vu, \vu}-p_{\vh, \vh}.
\end{align}
From the calculations in Section \ref{s31},
\begin{align}
-\mathcal{N}_f|\hat{\vh}|^2=&2\sum_{i,j=1,2}\ul{\hat h_i}\ul{\hat h_j}\partial_i\partial_jf
-\vN_f\cdot\ul{\nabla(|\hha|^2-\widehat{\mathcal{H}}_f^+ |\hha|^2)}
+(\widehat{\mathcal{N}}_f^+ -\mathcal{N}_f)|\hat{\vh}|^2,
\end{align}
see (\ref{mhd32}) and (\ref{mhd33}).\smallskip

{\bf Step 3.} The velocity equation.

Let
\begin{align*}
\vw = \p_t\vu+\vu\cdot\nabla\vu-\vh\cdot\nabla\vh+\nabla p.
\end{align*}
We will show that $\vw$ satisfies the following homogeneous equations:
\ben\label{eq:limit-u}
\left\{
\begin{array}{l}
\div\vw=0,\,\, \curl\vw=0\,\,\text{ in }\Omega_f,\\
\vw\cdot\vN_f=0\text{ on }\Gamma_f,\,\,
w_3=0\text{ on } \Gamma,\,\,\int_{\Gamma}w_i dx'=0, \,i=1,2,
\end{array}\right.
\een
which implies $\vw\equiv 0$, i.e.,
\begin{align*}
\p_t\vu+\vu\cdot\nabla\vu-\vh\cdot\nabla\vh+\nabla p = 0\,\,\text{ in } \Omega_f.
\end{align*}

First, by the definition of $p$,
\begin{align}\label{consis:div}
&\div (-\vu\cdot\nabla\vu+\vh\cdot\nabla\vh+\nabla p)=0,
\end{align}
which together with $\div \p_t\vu= 0$ yields $\div\vw=0$ in $\Om_f$. On the other hand, a direct computation by using the equation of $\vom$ shows
\begin{align*}\nonumber
\curl \p_t\vu =\p_t\curl\vu=\p_t\vom &=-\vu\cdot\nabla\vom+\vh\cdot\nabla\vj+\vom\cdot\nabla\vu-\vj\cdot\nabla\vh\\
&=\curl (-\vu\cdot\nabla\vu+\vh\cdot\nabla\vh+\nabla p).
\end{align*}
Thus, we obtain $\curl\vw=0$ in $\Omega_f$.

Since $u_3=0, h_3=0$ on $\Gamma$,
\begin{align}\label{eq:limit-wb}
w_3=\p_tu_3+u_i\partial_iu_3-h_i\partial_ih_3-\partial_3p=0\,\,\text{ on }\,\Gamma.
\end{align}
Moreover, according to (\ref{mhd72}),
\begin{align*}
\int_{\Gamma}w_i dx'=\int_{\Gamma}\big(\p_tu_i+u_j\partial_ju_i-h_j\partial_jh_i\big)dx'=0,\,\, i=1,2.
\end{align*}

It only leaves the boundary condition of $\vw$ on $\Gamma_f$ to be proved. To this end,
we first define the projection operator $\mathcal{P}:L^2(\BT)\to L^2(\BT)$ as
\begin{align*}
\mathcal{P}g = g-\langle g\rangle.
\end{align*}
By conversing the computations in Section \ref{s31}, we find
\begin{align*}
&\mathcal{P}\Big\{-2(\ul{u_1}\partial_1\theta+ \ul{u_2}\partial_2\theta)-\vN_f\cdot\ul{\nabla p}
-\sum_{i,j=1,2}\ul{u_i}\ul{u_j}\partial_i\partial_j f+\sum_{i,j=1,2}\ul{h_i}\ul{h_j}\partial_i\partial_jf\Big\}\\
=&-2\mathcal{P}(\ul{u_1}\partial_1\theta+\ul{u_2}\partial_2\theta)-\frac12\mathcal{N}_f|\hat\vh|^2-\mathcal{P}\vN_f\cdot\ul{\nabla(p_{\vu, \vu}-p_{\vh, \vh})}\\
& -\mathcal{P}\sum_{i,j=1,2}\ul{u_i}\ul{u_j}\partial_i\partial_jf+
\mathcal{P}\sum_{i,j=1,2}\ul{h_i}\ul{h_j}\partial_i\partial_jf\\
=&\mathcal{P}\Big\{-2(\ul{u_1}\partial_1\theta+\ul{u_2}\partial_2\theta\big)-\sum_{i,j=1,2}
\big(\ul{u_i}\ul{u_j}-\ul{h_i}\ul{h_j}
-\ul{\hat h_i}\ul{\hat h_j}\big)\partial_i\partial_jf\\
&-\vN_f\cdot\ul{\nabla(p_{\vu,\vu}-p_{\vh,\vh})}-\frac12\vN_f\cdot\ul{\nabla(|\hha|^2-\widehat{\mathcal{H}}_f^+ |\hha|^2)}
+\frac12(\widehat{\mathcal{N}}_f^+-\mathcal{N}_f)|\hat{\vh}|^2\Big\}\\
=&\mathcal{P}\p_t\theta.
\end{align*}
Recalling that $ \p_t\theta=\p_t^2f$, $\p_t f=\ul{\vu}\cdot\vN_f$, we obtain
\begin{align*}\nonumber
\mathcal{P}\Big\{\p_t(&\ul{\vu}\cdot\vN_f)+2\big(\ul{u_1}\partial_1
(\ul{\vu}\cdot\vN_f)+ \ul{u_2}\partial_2(\ul{\vu}\cdot\vN_f)\big)+\vN_f\cdot\ul{\nabla p}\\
&+\sum_{i,j=1,2}(\ul{h_i}\ul{h_j}-\ul{u_i}\ul{u_j})\partial_i\partial_jf\Big\}=0,
\end{align*}
which together with the fact
\[\p_t\vN_f=(-\partial_1\p_tf,-\partial_2\p_tf, 0)
=\Big(-\partial_1(\ul{\vu}\cdot\vN_f),-\partial_2(\ul{\vu}\cdot\vN_f), 0\Big)\]
implies
\begin{align*}
\mathcal{P}\Big\{(\ul{\p_t\vu}+\ul{\partial_3\vu}\p_tf)\cdot\vN_f
+\big(\ul{u_1}\partial_1(\ul{\vu}\cdot\vN_f)+\ul{u_2}\partial_2(\ul{\vu}\cdot\vN_f)\big)\\
+\sum_{i,j=1,2}\big(\ul{h_i}\ul{h_j}-\ul{u_i}\ul{u_j}\big)\partial_i\partial_jf+\vN_f\cdot\nabla p\Big\}=0.
\end{align*}
It follows from Lemma \ref{rel-uh} that
\begin{align}\label{mhd88}
\mathcal{P}\left\{(\p_t\vu+\vu\cdot\nabla\vu-\vh\cdot\nabla\vh+\nabla p)\left|_{\Gamma_f}\right.\cdot\vN_f\right\}=0.
\end{align}
On the other hand, (\ref{consis:div}) and (\ref{eq:limit-wb}) imply that
\[
\int_{\BT}(\p_t\vu+\vu\cdot\nabla\vu-\vh\cdot\nabla\vh+\nabla p)\left|_{\Gamma_f}\right.\cdot\vN_fdx'=0,
\]
which together with (\ref{mhd88}) yields
\begin{align*}
\vw\cdot\vN_f=(\p_t\vu+\vu\cdot\nabla\vu-\vh\cdot\nabla\vh+\nabla p)\cdot\vN_f=0\text{ on } \Gamma_f.
\end{align*}
This concludes the proof of (\ref{eq:limit-u}).\smallskip

{\bf Step 4.} The magnetic field equation.

Let $\vc{b}=\p_t\vh-\vh\cdot\nabla\vu+\vu\cdot\nabla\vh$.
It suffices to show that
\ben\label{eq:limit-H}
\left\{
\begin{array}{l}
\div \vc{b}=0,\,\, \curl\vc{b}=0\,\,\text{ in }\Omega_f,\\
\vc{b}\cdot\vN_f=0\,\text{ on }\Gamma_f,\,\,
b_3=0\,\text{ on } \Gamma,\,\int_{\Gamma}b_i dx'=0, \, i=1,2.
\end{array}\right.
\een

By using $\vh\cdot\vN_f=0$ on $\Gamma_f$, we get
\begin{align*}
0=&\p_t\big(\ul{\vh}\cdot\vN_f\big)+\ul{u_1}\partial_1(\ul{\vh}\cdot\vN_f)+ \ul{u_2}\partial_2(\ul{\vh}\cdot\vN_f)\\
=&\left(\ul{\p_t\vh} + \ul{\partial_3\vh}\p_tf\right)\cdot\vN_f+\ul{\vh}\cdot\p_t\vN_f
+\left(\ul{u_1}\ul{\partial_1\vh} + \ul{u_1}\ul{\partial_3\vh}\partial_1f\right)\cdot\vN_f\\
&\,+\ul{u_1}\ul{\vh}\cdot\partial_1\vN_f+ \left(\ul{u_2}\ul{\partial_2\vh}+ \ul{u_2}\ul{\partial_3\vh}\partial_2f\right)\cdot\vN_f + \ul{u_2}\ul{\vh}\cdot\partial_2\vN_f\\
=&\left(\ul{\p_t\vh}+\ul{\vu\cdot\nabla\vh}\right)\cdot\vN_f+\p_tf\ul{\partial_3\vh}\cdot\vN_f+\ul{\vh}\cdot\p_t\vN_f\\
&\,+\ul{u_1}\ul{\vh}\cdot\partial_1\vN_f + \ul{u_2}\ul{\vh}\cdot\partial_2\vN_f+\left(\ul{u_1}\ul{\partial_3
\vh}\partial_1f+\ul{u_2}\ul{\partial_3\vh}\partial_2f- \ul{u_3}\ul{\partial_3\vh}\right)\cdot\vN_f\\
=&\left(\ul{\p_t\vh}+\ul{\vu\cdot\nabla\vh}\right)\cdot\vN_f+ \ul{\vh}\cdot\p_t\vN_f
+\ul{u_1}\ul{\vh}\cdot\partial_1\vN_f +\ul{u_2}\ul{\vh}\cdot\partial_2\vN_f.
\end{align*}
On the other hand,
\begin{align*}
&\ul{\vh}\cdot\p_t\vN_f+ \ul{u_1}\ul{\vh}\cdot\partial_1\vN_f +\ul{u_2}\ul{\vh}\cdot\partial_2\vN_f\\
&=-\ul{h_1}\partial_1(\ul{\vu}\cdot\vN_f)-\ul{h_2}\partial_2(\ul{\vu}\cdot\vN_f)+\ul{u_1}\ul{\vh}\cdot\partial_1\vN_f +\ul{u_2}\ul{\vh}\cdot\partial_2\vN_f\\
&=-\ul{h_1}\left(\ul{\partial_1\vu}+\ul{\partial_3\vu}\partial_1f\right)\cdot\vN_f-
\ul{h_2}\left(\ul{\partial_2\vu}+\ul{\partial_3\vu}\partial_2f\right)\cdot\vN_f\\
&\quad -\ul{h_1}\ul{\vu}\cdot\partial_1\vN_f-\ul{h_2}\ul{\vu}\cdot\partial_2\vN_f-\sum_{i,j=1,2}\ul{u_i}\ul{h_j}\partial_i\partial_jf\\
&=\ul{h}_1\left(\ul{\partial_1\vu}+\ul{\partial_3\vu}\partial_1f\right)\cdot\vN_f-
\ul{h_2}\left(\ul{\partial_2\vu}+\ul{\partial_3\vu}\partial_2f\right)\cdot\vN_f\\
&=-(\ul{\vh\cdot\nabla\vu})\cdot\vN_f-(\ul{\partial_3\vu}\cdot\vN_f)(\ul{\vh}\cdot\vN_f)\\
&=-(\ul{\vh\cdot\nabla\vu})\cdot\vN_f.
\end{align*}
Thus, we deduce that
\begin{align}
\big(\p_t\vh-\vh\cdot\nabla\vu+\vu\cdot\nabla\vh\big)\cdot\vN_f=0\,\, \text{on}\,\, \Gamma_f.\nonumber
\end{align}
Moreover,
\begin{align*}
\div(\vh\cdot\nabla\vu-\vu\cdot\nabla\vh)=0
\end{align*}
together with $\div\p_t\vh = 0$ implies $\div\vc{b} = 0$. By (\ref{eq:limit-h}), we have
\begin{align*}
\curl(\p_t\vh)&=\p_t\vj=\curl(-\vu\cdot\nabla\vj+\vh\cdot\nabla\vom+\vj\cdot\nabla\vu
-\vom\cdot\nabla\vh-2\nabla u_i\times\nabla h_i)\\
&=\curl(\vh\cdot\nabla\vu-\vu\cdot\nabla\vh).
\end{align*}
Other boundary conditions in (\ref{eq:limit-H}) follows in a similar manner as in Step 3.

We remark that since $\hat\vh$ is a secondary variable solved out from $f$ (and $\hat{\vc{J}}$), it automatically satisfies the equations. Furthermore, stability condition (\ref{mhd12}) has been shown at the end of Section \ref{s61}. Hence, Step 1-Step 4 are enough to ensure that $(\vu,\vh,\hat\vh, f, p)$ obtained in Section \ref{s7} is a solution of the system (\ref{mhd01})-(\ref{mhd7}).

\section*{Acknowledgment}
Yongzhong Sun is supported by NSF of China under Grant No. 11571167 and the PAPD of Jiangsu Higher Education Institutions. Wei Wang is supported by NSF of China under Grant No. 11501502 and the Fundamental Research Funds for the Central Universities.
Zhifei Zhang is partially supported by NSF of China under Grant No. 11425103.

\end{document}